\documentclass[a4paper]{amsart}
\usepackage{amsmath}
\usepackage{amsthm}
\usepackage{amssymb}
\usepackage{enumerate}
\usepackage{latexsym}
\usepackage{empheq}
\usepackage{graphicx}
\usepackage{tikz}
\newtheorem{Thm}{Theorem}[section]
\newtheorem{Prop}[Thm]{Proposition}
\newtheorem{Cor}[Thm]{Corollary}
\newtheorem{Lem}[Thm]{Lemma}

\newtheorem{Cha}{Property A}[section]
\newtheorem{Ch}{Property B}[section]
\newtheorem{A}{Theorem A}[section]
\newtheorem{B}{Theorem B}[section]
\newtheorem{C}{Theorem C}[section]
\newtheorem{formula}{Formulas}[section]

\theoremstyle{definition}
\newtheorem{notation}{Notation}
\newtheorem{Def}{Definition}
\theoremstyle{remark}
\newtheorem{Rem}{Remark}[section]
\newtheorem{Example}{Example}[section]
\newcommand{\Span}{\mathop{\mathrm{Span}}\nolimits}
\newcommand{\Ric}{\mathop{\mathrm{Ric}}\nolimits}
\newcommand{\Spec}{\mathop{\mathrm{Spec}}\nolimits}
\newcommand{\supp}{\mathop{\mathrm{supp}}\nolimits}
\newcommand{\tr}{\mathop{\mathrm{tr}}\nolimits}
\newcommand{\Met}{\mathop{\mathrm{Met}}\nolimits}
\newcommand{\Le}{\mathop{\mathrm{L}}\nolimits}
\newcommand{\vol}{\mathop{\mathrm{vol}}\nolimits}
\newcommand{\dv}{\mathop{\mathrm{div}}\nolimits}
\newcommand{\Ker}{\mathop{\mathrm{Ker}}\nolimits}
\newcommand{\Vol}{\mathop{\mathrm{Vol}}\nolimits}

\title[Rotational Symmetry of the Standard Sphere]{Riemannian Invariants that Characterize Rotational Symmetries of the Standard Sphere}
\author{Masayuki Aino}
\address{Graduate School of Mathematics, Nagoya University, Chikusa-Ku Nagoya, 464-8602, Japan}
\email{m16100c@math.nagoya-u.ac.jp}
\begin{document}

\maketitle
\begin{abstract}
Inspired by the Lichnerowicz-Obata theorem for the first eigenvalue of the Laplacian, we define a new family of invariants $\{\Omega_k(g)\}$ for closed Riemannian manifolds.
The value of $\Omega_k(g)$ sharply reflects the  spherical part of the manifold.
Indeed, $\Omega_1(g)$ and $\Omega_2(g)$ characterize the standard sphere.
\end{abstract}
\tableofcontents
\section{Introduction}

In this paper we introduce a new family of Riemannian invariants that characterizes the standard sphere.
Indeed, we define $\Omega_1 \geq \Omega_2\geq \cdots \to 0$ for closed Riemannian manifolds, and show that $\Omega_1$ and $\Omega_2$ characterize the standard sphere.
Before explaining our result in detail, we provide historical backgrounds.

Let $(M,g)$ be a closed Riemannian manifold of dimension $n$ and $\Spec(M,g)=\{0=\lambda_0< \lambda_1 \leq \lambda_2 \leq\cdots \to \infty\}$ the set of eigenvalues of the Laplacian $\Delta=-g^{ij}\nabla_i \nabla_j$ acting on $C^\infty(M)$.
We can construct a complete orthonormal system $\{\psi_i\}$ in $L^2(M)$ of eigenfunctions of $\Delta$:
\begin{equation*}
\Delta \psi_i=\lambda_i \psi_i.
\end{equation*}

Eigenvalues characterize the standard sphere in some cases.
For example, Tanno \cite[Theorem B]{Tan} showed that, for $n\leq 6$, $\Spec(S^n(c))$ completely characterizes the standard $n$-dimensional sphere $S^n(c)$ of radius $c$, i.e., if $\Spec(M,g)=\Spec(S^n(c))$, then $(M,g)$ is isometric to $S^n(c)$.
Under the condition $\Ric \geq r g$ for a positive constant $r>0$, Lichnerowicz \cite{Li} showed
\begin{equation*}
\lambda_1\geq \frac{n}{n-1}r,
\end{equation*}
and Obata showed the equality holds if and only if $(M,g)$ is isometric to the standard sphere of radius $\sqrt{(n-1)/r}$.
In the proof, the following theorem \cite[Theorem A]{Ob} plays an important role.
\begin{Thm}
Let $c$ be a positive constant. If $(M,g)$ admits a non-constant function $u\in C^\infty(M)$ with $\nabla^2 u=-c^2 u g$, then $(M,g)$ is isometric to $S^n(1/c)$.
\end{Thm}
\noindent We emphasize that this theorem needs no assumption about Ricci curvature. 
Similarly, Tashiro \cite[Lemma2.2]{T1} showed the following theorem.
\begin{Thm}\label{Tas}
If $(M,g)$ admits a non-constant function $u \in C^\infty(M)$ with $\nabla^2 u=-\frac{\Delta u}{n} g$, then $(M,g)$ is conformal to $S^n(1)$.
\end{Thm}
\noindent This theorem plays an important role in Theorem A below.
We review Tashiro's work in section 3.

In this paper, inspired by these results, we define non-negative real numbers $\Omega_1(g)\geq \Omega_2(g)\geq \Omega_3(g)\geq \cdots \to 0$ for a closed Riemannian manifold $(M,g)$, and study their properties. The value of $\Omega_1(g)$ is defined by
\begin{equation*}
\Omega_1(g)=\sup \left\{\int_M \Ric(\nabla v,\nabla v)\,d\mu_g:v\in L^2_2(M) \text{ and } \|\Delta v\|_{L^2}=1\right\}.
\end{equation*}
We define $\{\Omega_k(g)\}_{k\geq2}$ in Definition \ref{def1} and Definition \ref{def2}. 
If $\Omega_1(g)>0$, then we have $\Omega_k(g)>0$ for any $k\in \mathbb{Z}_{>0}$, and there exists a non-constant function $v_k\in C^\infty(M)$ such that $\Delta^2 v_k=\frac{1}{\Omega_k(g)}\nabla^\ast(\Ric(\nabla v_k,\cdot))$.
We have the following two theorems:
\begin{A}
For any closed $n$-dimensional Riemannian manifold $(M,g)$, we have $\Omega_1(g) \leq \frac{n-1}{n}$.
If $\Omega_1(g)=\frac{n-1}{n}$, then $(M,g)$ is conformal to the standard sphere.
\end{A} 
\begin{B}
Let $(M,g)$ be a closed Riemannian manifold of dimension $n$.
If $\Omega_1(g)=\Omega_2(g)=\frac{n-1}{n}$, then $(M,g)$ is isometric to the standard sphere of a certain radius.
\end{B}
\noindent Consequently, $\Omega_1(g)$ and $\Omega_2(g)$ characterize the $n$-dimensional standard sphere.
Roughly speaking, $\Omega_1(g)=\frac{n-1}{n}$ means that $(M,g)$ has a rotational symmetry.
Moreover, $\Omega_2(g)=\frac{n-1}{n}$ means that $(M,g)$ has another rotation axis.
Theorem B results from the fact that no Riemannian manifold except for the standard sphere has two distinct rotation axes (for the explicit statement, see Lemma \ref{clam}). 

The relationship between Theorem A and Lichnerowicz inequality is the following:
\begin{C}
If there exists a constant $r>0$ such that $\Ric \geq r g$, then
\begin{equation*}
\lambda_1 \geq\frac{r}{\Omega_1(g)}\geq \frac{n}{n-1}r.
\end{equation*}
\end{C}

In Theorem \ref{sph}, we see that the value of $\Omega_1(g)$ sharply reflects the spherical part of the manifold.
In section 4, we give some computation and examples.

In the next section, we define a family of non-negative numbers $\{\Lambda_k(S)\}_{k\in \mathbb{Z}_{>0}}$ for a closed Riemannian manifold $(M,g)$ and a symmetric tensor $S$ of type $(0,2)$ on $M$. In section 3, we specialize the case when $S=\Ric_g$, and $\Lambda_k(\Ric_g)$ recovers $\Omega_k(g)$ above.
\begin{sloppypar}
{\bf Acknowledgments}.\ 
I am grateful to my supervisor, Professor Shinichiroh Matsuo, for his useful comments and advice.
\end{sloppypar}

\section{The definition of $\Lambda_k(S)$ and elementary properties}

\subsection{Preliminaries}
\begin{notation}
Let \((M,g)\) be a closed Riemannian manifold of dimension $n$.
$H$ denotes the closed subspace of $L^2_2(M)$ defined by
\begin{equation*}
H=H_{(M,g)}=\{v\in L^2_2(M):\int_M v\,d\mu_g=0 \},
\end{equation*}
with an inner product
\begin{equation*}
\langle u,v\rangle_H=\int_M (\Delta u,\Delta v)\,d\mu_g \quad(u,v\in H).
\end{equation*}
\end{notation}
The norm induced by this inner product is equivalent to the norm induced by the standard inner product on $L^2_2(M)$.
Note that $\Ker \Delta |_H=\{0\}$, and so there exists a constant $C>0$ such that $\|v\|_{L^2_2}\leq C\|\Delta v\|_{L^2}$ for all $v\in H$.
This is what is called the elliptic estimate.

\begin{notation}
Let $S$ be a symmetric tensor of type $(0,2)$ on $M$. 
We define a functional $\Lambda_S \colon H\backslash \{0\} \to \mathbb{R}$ by
\begin{equation*}
\Lambda_S (v)=\frac{\int_M S(\nabla v,\nabla v)\,d\mu_g}{\|\Delta v\|^2_{L^2}}.
\end{equation*}
\end{notation}
By the elliptic estimate, we have $\sup_{v\in H\backslash \{0\} } \Lambda_S (v)<\infty$.
Note that, for each $v\in L^2_2(M)$, there exists a constant $c\in \mathbb{R}$ such that $v-c\in H$, and so
\begin{equation*}
\sup_{v\in H\backslash \{0\} } \Lambda_S (v)= \sup \left\{\frac{\int_M S(\nabla v,\nabla v)\,d\mu_g}{\|\Delta v\|^2_{L^2}}:v\in L^2_2(M) \text{ and } v \text{ is not a constant}\right\}.
\end{equation*}

\begin{Lem}\label{fund1}
Let $F\neq \{0\}$ be a closed subspace of $H$. 
If $ \sup_{v\in F\backslash \{0\} } \Lambda_S (v)> 0$, then there exists a function $u \in F \backslash \{0\}$ such that $\sup_{v\in F\backslash \{0\} } \Lambda_S (v)= \Lambda_S(u)$.
\end{Lem}
\begin{proof}
Put $\Lambda=\sup_{v\in F\backslash \{0\} } \Lambda_S (v)$.
Let $\{u_k\}$ be a sequence of $F$ such that $\Lambda=\lim_{k\to \infty} \Lambda_S(u_k)$ and $\|\Delta u_k\|_{L^2}=1$. 
By the elliptic estimate, there exists a constant $C>0$ such that $\|u_k\|_{L^2_2} \leq C$ for all $k$.
Thus, we can take a subsequence of $\{u_k\}$ (denote it again by $\{u_k\}$) and $u\in F$ such that
\begin{enumerate}
\item[] $u_k \to u$ (strongly in $L^2_1(M)$),
\item[] $u_k \rightharpoonup u$ (weakly in $F$).
\end{enumerate}
Then,
\begin{equation}\label{conv}
\begin{split}
&\Lambda-\Lambda_S(u_k)\\
=&\Lambda \int_M (\Delta u_k,\Delta u_k)\,d\mu_g - \int_M S(\nabla u_k,\nabla u_k)\,d\mu_g\\
=&\Lambda \int_M (\Delta u + \Delta(u_k -u),\Delta u + \Delta(u_k -u))\,d\mu_g\\
&\qquad - \int_M S(\nabla u,\nabla u)\,d\mu_g+\epsilon_k\\
=&\Lambda
 \|\Delta u\|^2_{L^2}+2\Lambda \int_M (\Delta (u_k - u),\Delta u)\,d\mu_g+\Lambda \|\Delta (u_k-u)\|^2_{L^2}\\
&\qquad - \int_M S(\nabla u,\nabla u)\,d\mu_g+ \epsilon_k\\
\geq& 2\Lambda \int_M (\Delta (u_k - u),\Delta u)\,d\mu_g+\Lambda \|\Delta (u_k-u)\|^2_{L^2}+\epsilon_k,
\end{split}
\end{equation}
where we put $\epsilon_k=\int_M S(\nabla u,\nabla u)\,d\mu_g-\int_M S(\nabla u_k,\nabla u_k)\,d\mu_g$.
In the last line, we have used $\int_M S(\nabla u,\nabla u)\,d\mu_g\leq \Lambda \|\Delta u\|^2_{L^2}$.
By (\ref{conv}), we have
\begin{equation}\label{est}
\begin{split}
&\Lambda \|\Delta (u_k-u)\|^2_{L^2}\\
\leq &\Lambda-\Lambda_S(u_k)-2\Lambda \int_M (\Delta (u_k - u),\Delta u)\,d\mu_g-\epsilon_k.
\end{split}
\end{equation}
The right hand side of (\ref{est}) converges to $0$ as $k\to \infty$. Since we assumed $\Lambda >0$, we get
\begin{equation*}
\lim \|\Delta (u_k-u)\|_{L^2}=0.
\end{equation*}
This implies that $u_k$ converges to $u$ strongly in $H$, and so $\Lambda = \Lambda_S(u)$.   
\end{proof}

Now, we consider the case when $F=H$.
\begin{Lem}\label{fund2}
Put $\Lambda_1(S)=\sup_{v\in H\backslash \{0\} } \Lambda_S (v)$.
If $\Lambda_1(S)>0$, for $u\in H\backslash \{0\} $, the following two conditions are equivalent:
\begin{itemize}
\item[(a)] $\Lambda_1(S)=\Lambda_S(u)$,
\item[(b)] $\Delta^2 u=\frac{1}{\Lambda_1(S)}\nabla^\ast (S(\nabla u,\cdot))$.
\end{itemize}
In particular, if $u$ satisfies $\Lambda_1(S)=\Lambda_S(u)$, then $u$ is smooth.
\end{Lem}
\begin{proof}
We first prove (a) $\Rightarrow$ (b).
Suppose that $u\in H\backslash \{0\} $ satisfies $\Lambda_1(S)=\Lambda_S(u)$, i.e.,
\begin{equation*}
\Lambda_1(S)=\sup_{v\in H\backslash \{0\}}\frac{\int_M S(\nabla v,\nabla v)\,d\mu_g}{\|\Delta v\|^2_{L^2}}=\frac{\int_M S(\nabla u,\nabla u)\,d\mu_g}{\|\Delta u\|^2_{L^2}}.
\end{equation*}
Then, for any $v \in L^2_2(M)$, we have
\begin{equation*}
\Lambda_1(S)\int_M (\Delta (u+tv),\Delta (u+tv))\,d\mu_g
-\int_M S(\nabla (u+tv),\nabla (u+tv))\,d\mu_g \geq 0 \quad (t\in \mathbb{R}),
\end{equation*}
\begin{equation*}
\Lambda_1(S)\int_M (\Delta u,\Delta u)\,d\mu_g - \int_M S(\nabla u,\nabla u)\,d\mu_g = 0.
\end{equation*}
Therefore,
\begin{equation*}
\begin{split}
0
=&\left. \frac{d}{dt}\right|_{t=0} \Bigl(\Lambda_1(S)\int_M (\Delta (u+tv),\Delta (u+tv))\,d\mu_g\\
&\qquad\qquad\quad -\int_M S(\nabla (u+tv),\nabla (u+tv))\,d\mu_g\Bigr)\\
=&2\Lambda_1(S)\int_M (\Delta u,\Delta v)\,d\mu_g - 2\int_M S(\nabla u,\nabla v)\,d\mu_g\\
=&2\Lambda_1(S)\int_M \left(\Delta^2 u-\frac{1}{\Lambda_1(S)}\nabla^\ast (S(\nabla u,\cdot)),v\right)\,d\mu_g.
\end{split}
\end{equation*}
Since the pairing $L^2_{-2}(M)\times L^2_2(M)\to \mathbb{R}$ is non-degenerate, we conclude $\Delta^2 u-\frac{1}{\Lambda_1(S)}\nabla^\ast (S(\nabla u,\cdot))=0$.

We next prove (b) $\Rightarrow$ (a).
Suppose that $\Delta^2 u=\frac{1}{\Lambda_1(S)}\nabla^\ast (S(\nabla u,\cdot))$.
Then we have,
\begin{equation*}
\int_M S(\nabla u,\nabla u)\,d\mu_g=\int_M (\nabla^\ast (S(\nabla u,\cdot)),u)\,d\mu_g  = \Lambda_1(S) \int_M (\Delta u,\Delta u)\,d\mu_g.
\end{equation*}
Therefore, we get $\Lambda_1(S)= \frac{\int_M S(\nabla u,\nabla u)\,d\mu_g}{\|\Delta u\|^2_{L^2}}=\Lambda_S(u)$.   
\end{proof}

Next we consider the sign of $\sup_{v\in F\backslash \{0\} } \Lambda_S (v)$.

\begin{Lem}\label{lem}
For any symmetric tensor $S$ of type $(0,2)$ and any infinite dimensional closed subspace $F\subset H$, we have $\sup_{v\in F\backslash \{0\} } \Lambda_S (v)\geq 0$.
Moreover, if $S\leq 0$, then $\sup_{v\in F\backslash \{0\} } \Lambda_S (v)=0$ holds.
\end{Lem}
\begin{proof}
Since $M$ is compact, there exists a constant $c>0$ such that $S\geq -cg$.
Let
\begin{equation*}
0=\lambda_0 <\lambda_1\leq \lambda_2 \leq \cdots \to \infty
\end{equation*}
be the eigenvalues of the Laplacian $\Delta$ and $\{ \psi_i \}$ a complete orthonormal system in $L^2(M)$ of the eigenfunctions of $\Delta$:
\begin{equation*}
\Delta \psi_i= \lambda_i \psi_i.
\end{equation*}

Since $F\subset H$ is infinite dimensional, for any $k=0,1,2,\ldots$, we can take $v_k \in F\backslash \{0\}$ that is orthogonal to $\psi_0,\ldots,\psi_k$ in $L^2(M)$.
Then,
\begin{equation*}
\begin{split}
\int_M S(\nabla v_k,\nabla v_k)\,d\mu_g
&\geq -c \int_M (\nabla v_k,\nabla v_k)\,d\mu_g\\
&\geq -\frac{c}{\lambda_{k+1}}\|\Delta v_k\|^2_{L^2}.
\end{split}
\end{equation*}
Therefore, $\sup_{v\in F\backslash \{0\} } \Lambda_S (v)\geq \Lambda_S(v_k)\geq -\frac{c}{\lambda_{k+1}}$.
The right hand side converges to $0$ as $k\to \infty$, and so we get  $\sup_{v\in F\backslash \{0\} } \Lambda_S (v)\geq0$.

Moreover, if $S\leq0$, we have $\Lambda_S(v) \leq 0$ for any $v\in F\backslash \{0\}$. Therefore, $\sup_{v\in F\backslash \{0\} } \Lambda_S (v)\leq 0$.
Thus, we get $\sup_{v\in F\backslash \{0\} } \Lambda_S (v)=0$.  
\end{proof}

In the case when $F=H$, the converse is also true:
\begin{Lem}\label{one}
If $\sup_{v\in H\backslash \{0\} } \Lambda_S (v)=0$, then we have \(S\leq 0\).
\end{Lem}

Moreover, we can show the following:
\begin{Lem}\label{subsp}
If $S\leq0$ does not hold, for any positive integer $k\in \mathbb{N}$, there exists a $k$-dimensional subspace $V_k$ of $C^\infty(M)\cap H$ that satisfies the following: $\Lambda_S(v)>0$ holds for any $v \in V_k \backslash \{0\}$.
\end{Lem}
Putting $k=1$ in Lemma \ref{subsp}, we get Lemma \ref{one}.
See Appendix for the proof of Lemma \ref{subsp}.

\subsection{The definition and elementary properties of $\Lambda_k(S)$}
Let $(M,g)$ be a closed Riemannian manifold of dimension $n$ and $S$ a symmetric tensor of type $(0,2)$ on $M$.
We define $\Lambda_1(S)\geq \Lambda_2(S)\geq \cdots$.
\begin{Def}\label{def1}
We put $\Lambda_1(S)=\sup_{v\in H\backslash \{0\} } \Lambda_S (v)$.
If $\Lambda_1(S)=0$, we define $\Lambda_1(S)=\Lambda_2(S)=\Lambda_3(S)=\cdots=0$.
If $\Lambda_1(S)>0$, we can take a function $v_1\in H$ such that $\Lambda_S(v_1)=\Lambda_1(S)$ and $\|\Delta v_1\|_{L^2}=1$ by Lemma \ref{fund1}.
We define $\Lambda_k(S)$ and $v_k$ inductively as follows.
Suppose that we have chosen $\Lambda_1(S),\cdots,\Lambda_k(S)>0$ and $v_1,\cdots,v_k\in H$.
We put 
\begin{equation*}
\Lambda_{k+1}(S)=\sup \left\{\Lambda_S (v):v\in\langle v_1,\cdots,v_k\rangle^{\perp}\backslash \{0\}\right\},
\end{equation*}
where $\langle v_1,\cdots,v_k\rangle^{\perp}$ denotes the orthogonal complement of $\Span_{\mathbb{R}} \{v_1,\cdots,v_k\}$ in $H$.
We have $\Lambda_{k+1}(S)>0$ by Lemma \ref{lem} and Lemma \ref{subsp};
therefore, we can take $v_{k+1} \in \langle v_1,\cdots,v_k\rangle^{\perp}\backslash \{0\}$ such that $\Lambda_{k+1}(S) = \Lambda_S(v_{k+1})$ and $\|\Delta v_{k+1}\|_{L^2}=1$ by Lemma \ref{fund1}.
Inductively, we define $\Lambda_1(S)\geq \Lambda_2(S)\geq \cdots >0$ and $v_1,v_2,\cdots \in H$.
We call $v_k$ the {\bf associated function} to $\Lambda_k(S)$.
\end{Def}
We can prove each $v_k$ is smooth. In fact, we show the following:
\begin{Prop}\label{fund3}
If $\Lambda_1(S)>0$, then we have $\Delta^2 v_i=\frac{1}{\Lambda_i(S)}\nabla^\ast (S(\nabla v_i,\cdot))$ for each $i=1,2,\ldots$.
\end{Prop}
To prove Proposition \ref{fund3}, we show the following lemma.
\begin{Lem}\label{3lem}
Take a real number $\Lambda\in \mathbb{R}$ and a function $v\in H$.
Suppose that $\Lambda \Delta^2 v=\nabla^\ast (S(\nabla v,\cdot))$ holds.
Then, we have
\begin{equation*}
\int_M S(\nabla v,\nabla u) \,d\mu_g=0
\end{equation*}
for any function $u\in H$ with $\langle u, v\rangle_H$=0.
\end{Lem}
\begin{proof}
For any function $u\in H$ with $\langle u,v\rangle_H=0$, we have
\begin{align*}
\int_M S(\nabla v,\nabla u)\,d\mu_g
=&\int_M (\nabla^\ast (S(\nabla v,\cdot)),u)\,d\mu_g\\
=&\Lambda \int_M (\Delta^2 v,u) \,d\mu_g\\
=&\Lambda \int_M (\Delta v,\Delta u) \,d\mu_g=0.
\end{align*}
\end{proof}
\begin{proof}[Proof of Proposition \ref{fund3}]
We show the proposition by induction on $i$.
We have shown the proposition for $i=1$ in Lemma \ref{fund1}.
Suppose that we have shown the proposition for $i=1,\ldots,k$.
Then, for any $v\in \langle v_1,\ldots,v_{k}\rangle^\perp$, similarly to Lemma \ref{fund2}, we have
\begin{equation}\label{eqk}
\int_M \Bigl(\Delta^2 v_{k+1}-\frac{1}{\Lambda_{k+1}(S)}\nabla^\ast (S(\nabla v_{k+1},\cdot)),v\Bigr) \,d\mu_g=0.
\end{equation}
For each $v_i \,(i=1,\cdots,k)$, we have $\langle v_{k+1},v_{i}\rangle_H=0$.
Therefore, by Lemma \ref{3lem},
\begin{equation*}
\int_M \Bigl(\nabla^\ast (S(\nabla v_{k+1},\cdot)),v_i\Bigr)\,d\mu_g=0.
\end{equation*}
Thus,
\begin{equation*}
\int_M \Bigl(\Delta^2 v_{k+1}-\frac{1}{\Lambda_{k+1}(S)}\nabla^\ast (S(\nabla v_{k+1},\cdot)),v_i\Bigr) \,d\mu_g=0.
\end{equation*}
Consequently, (\ref{eqk}) holds for any $v\in L^2_2(M)$. This implies the proposition for $i=k+1$.  
\end{proof}

We introduce some elementary properties of $\Lambda_k(S)$.
\begin{Prop}\label{lambda0}
For any symmetric tensor $S$ of type $(0,2)$, we have $\lim_{k\to \infty} \Lambda_k(S)=0$.
\end{Prop}
\begin{proof}
We assume $\Lambda_1(S)>0$ (otherwise the lemma is trivial).

We put $\Lambda =\lim_{k\to \infty} \Lambda_k(S)$.
Let $H^+(S)=\overline{\bigoplus_{k=1}^{\infty} \mathbb{R}v_k}$ be the closure of $\bigoplus_{k=1}^{\infty} \mathbb{R}v_k$ in $H$.
For any $v\in H^+(S)\backslash \{0\}$, we have 
\begin{equation*}
v=\lim_{l\to \infty}\sum_{k=1}^l \langle v,v_k\rangle v_k\quad (\text{strongly in }H).
\end{equation*}
Therefore, by Lemma \ref{3lem}
\begin{equation*}
\begin{split}
\Lambda_S(v)&=\frac{\int_M S(\nabla v,\nabla v)\,d\mu_g}{\|\Delta v\|^2_{L^2}}\\
&=\frac{\sum \langle v,v_k\rangle^2\int_M S(\nabla v_k,\nabla v_k)\,d\mu_g}{\sum \langle v,v_k\rangle^2\|\Delta v_k\|^2_{L^2}}\\
&\geq \frac{\sum \langle v,v_k\rangle^2 \Lambda \|\Delta v_k\|^2_{L^2}}{\sum \langle v,v_k\rangle^2\|\Delta v_k\|^2_{L^2}}
= \Lambda.
\end{split}
\end{equation*}
Then, for any $v\in H^+(S)\backslash \{0\}$, we have $\Lambda_{-S}(v)\leq -\Lambda$, and so $\sup_{v\in H^+(S)\backslash \{0\} } \Lambda_{-S} (v) \leq -\Lambda$.
By Lemma \ref{lem}, we have $\sup_{v\in H^+(S)\backslash \{0\} } \Lambda_{-S} (v)\geq 0$; therefore, $-\Lambda \geq 0$.
However, we have $\Lambda\geq0$.
Thus, $\Lambda=0$ holds.
\end{proof}
\begin{Lem}\label{perp}
Take real numbers $\Lambda,\overline{\Lambda}\in\mathbb{R}$ and functions $v,\overline{v} \in H$.
If $\Lambda\neq\overline{\Lambda}$ and
\begin{align*}
\Lambda\Delta^2 v&=\nabla^\ast (S(\nabla v,\cdot))\\
\overline{\Lambda}\Delta^2\overline{v}&=\nabla^\ast (S(\nabla \overline{v},\cdot)),
\end{align*}
then $\langle v,\overline{v}\rangle_H=0$.
\end{Lem}
\begin{proof}
\begin{align*}
\Lambda \int_M (\Delta v,\Delta \overline{v})\,d\mu_g
=&\int_M \left(\nabla^\ast (S(\nabla v,\cdot)),\overline{v}\right)\,d\mu_g\\
=&\int_M \left(v,\nabla^\ast (S(\nabla \overline{v},\cdot))\right)\,d\mu_g\\
=&\overline{\Lambda} \int_M (\Delta v,\Delta \overline{v})\,d\mu_g.
\end{align*}
Thus, we get $\int_M (\Delta v,\Delta \overline{v})\,d\mu_g=0$.
\end{proof}
We define $\Lambda_{-i}(S)=-\Lambda_i(-S)$.
If $-\Lambda_{-1}(S)>0$, we define $v_{-i}$ to be the associated function to $\Lambda_{-i}(S)$, i.e., $v_{-1},v_{-2},\cdots$ are orthonormal in $H$, and $\Delta^2 v_{-i}=\frac{1}{\Lambda_{-i}(S)}\nabla^\ast (S(\nabla v_{-i},\cdot))$ holds.
We define 
\begin{enumerate}
\item[]$H^+(S)=\overline{\bigoplus_{k=1}^{\infty} \mathbb{R}v_k}$,
\item[]$H^-(S)=H^+(-S)=\overline{\bigoplus_{k=1}^{\infty} \mathbb{R}v_{-k}}$,
\end{enumerate}
where the overline means the closure in $H$ (if $\Lambda_1(S)=0$, then we define $H^+(S)=\{0\}$).
Then, $H^+(S)$ and $H^-(S)$ are orthogonal to each other in $H$ by Lemma \ref{perp}.
Let $H^0(S)$ be the orthogonal complement of $H^+(S)\oplus H^-(S)$ in $H$.
We can characterize the element of $H^0(S)$ by the following lemma.
\begin{Prop}\label{orth}
For each $v\in H$ the following two conditions are equivalent:
\begin{itemize}
\item[(a)] $v\in H^0(S)$,
\item[(b)] $\nabla^\ast (S(\nabla v,\cdot))=0$.
\end{itemize}
\end{Prop}
\begin{proof}
We first prove (a)$\Rightarrow$(b).
Take $v\in H^0(S)$.
We assume $v\neq0$ (otherwise the lemma is trivial).
For any $k=1,2,\ldots$, we have $v \in \langle v_1,\cdots,v_k\rangle^\perp$.
Thus, $\Lambda_S(v)\leq \Lambda_{k+1}(S)$.
Taking the limit, we get $\Lambda_S(v) \leq 0$ by Proposition \ref{lambda0}.
Similarly, we have $\Lambda_{-S}(v)\leq 0$; therefore, $\Lambda_S(v)=0$.
Thus, for any $w\in H^0(S)$ and $t\in \mathbb{R}$, we have $\int_M S(\nabla (v+tw),\nabla (v+tw))\,d\mu_g=0$.
Therefore, we get
\begin{equation}\label{eq0}
\int_M S(\nabla v,\nabla w)\,d\mu_g=0.
\end{equation}

By Lemma \ref{3lem}, for any $k=1,2,\ldots$,
\begin{equation*}
\int_M S(\nabla v,\nabla v_k)\,d\mu_g=\int_M S(\nabla v,\nabla v_{-k})\,d\mu_g=0.
\end{equation*}
This implies that the equation (\ref{eq0}) holds for all $w \in H$.
Therefore, we get $\nabla^\ast (S(\nabla v,\cdot))=0$.

We next prove (b) $\Rightarrow$ (a).
Suppose $\nabla^\ast (S(\nabla v,\cdot))=0$.
Then, $v$ is orthogonal to $H^+(S)$ and $H^-(S)$ by Lemma \ref{perp}.
Thus, we get (a).
\end{proof}

Next, we characterize $v_k$.
\begin{Lem}\label{vkch}
Take $u\in H\backslash\{0\}$. Suppose that there exists a constant $a\in \mathbb{R}$ such that $\Delta^2 u=a \nabla^\ast (S(\nabla u,\cdot))$.
Then, there exists a non-zero integer $k \in \mathbb{Z} \backslash \{0\}$ such that $a=\frac{1}{\Lambda_k(S)}$ and $u\in \Span_\mathbb{R}\{v_i:\Lambda_S(v_i)=\Lambda_k(S)\}$.
\end{Lem}
\begin{proof}
Since $\Delta^2 u \neq 0$, we have $a \neq0$.
Then, $u$ is orthogonal to $H^0(S)$ and $v_i$ for $i\in \mathbb{Z}\backslash\{0\}$ such that $\frac{1}{\Lambda_i(S)}\neq a$ by Lemma \ref{perp}.
Therefore, there exists a non-zero integer $k\in \mathbb{Z}\backslash \{0\}$ such that $\frac{1}{\Lambda_k(S)}=a$ and $u\in \Span_\mathbb{R}\{v_i:\Lambda_S(v_i)=\Lambda_k(S)\}$.
\end{proof}
Lemma \ref{vkch} immediately implies the following two corollaries.
\begin{Cor}\label{koyuti1}
Suppose that $\{u_l\}_{l=1}^\infty$ is a complete orthonormal system in $H$, and, for each $l\in \mathbb{Z}_{>0}$, there exists $c_l\in \mathbb{R}$ such that $c_l \Delta^2 u_l=\nabla^\ast (S(\nabla u_l,\cdot))$. 
Then,
\begin{equation*}
H^+(S)=\overline{\bigoplus_{l\in \mathbb{Z}_{>0},c_l>0} \mathbb{R}u_l},\quad H^0(S)=\overline{\bigoplus_{l\in \mathbb{Z}_{>0},c_l=0} \mathbb{R}u_l},
\quad H^-(S)=\overline{\bigoplus_{l\in \mathbb{Z}_{>0},c_l<0} \mathbb{R}u_l}.
\end{equation*}
Moreover, if $\Lambda_1(S)>0$, we have
\begin{equation*}
\{\Lambda_k(S):k\in\mathbb{Z}_{>0}\}=\{c_l:l\in\mathbb{Z}_{>0}\text{ and }c_l>0\}.
\end{equation*}
\end{Cor}
\begin{Cor}\label{koyuti2}
Let $0<\lambda_1 \leq \lambda_2 \leq \cdots \to \infty$ be the eigenvalues of the Laplacian.
Then, we have $\Lambda_k(g)=\frac{1}{\lambda_k}$ for all $k\in\mathbb{Z}_{>0}$.
\end{Cor}
\subsection{Other properties of $\Lambda_k(S)$}
We consider the value of $\Lambda_1$ of the product of Riemannian manifolds.
Let $(M_i,g_i)$ $(i=1,2)$ be $n_i$-dimensional compact Riemannian manifolds, and $S_i$ symmetric tensors of type $(0,2)$ on $M_i$.
We denote the projections by $\pi_i \colon M_1 \times M_2 \to M_i$.
Then, we have $T(M_1 \times M_2)=\pi_1^\ast TM_1 \bigoplus \pi_2^\ast TM_2$.
We consider the product metric $g=g_1 + g_2$ on $M=M_1 \times M_2$.
We define a symmetric tensor of type $(0,2)$ on $M$ by $S=S_1\oplus S_2$.
Then, we have the following proposition.
\begin{Prop}\label{prodmf}
We have $\Lambda_1(S)=\max \{\Lambda_1(S_1),\Lambda_1(S_2)\}$.
\end{Prop}
\begin{proof}
Let $\{\lambda_i\}$ (resp. $\{\lambda'_k\}$) be the eigenvalues and $\{\psi_i\}$ (resp. $\{\psi'_k\}$) the eigenfunctions of the Laplacian of $(M_1,g_1)$ (resp. $(M_2,g_2)$).
For any $u \in C^\infty(M)$, we have the following decomposition:
\begin{equation*}
u(x,y)=\sum a_{ik} \psi_i(x)\psi'_k(y)=\sum b_k(x) \psi'_k(y)=\sum c_i(y)\psi_i(x). 
\end{equation*}

Therefore,
\begin{equation*}
\begin{split}
S(\nabla u,\nabla u)
&=\sum a_{ik}a_{jl}\Bigl(S_1(\nabla^{g_1}\psi_i,\nabla^{g_1}\psi_j)\psi'_k\psi'_l+\psi_i\psi_j S_2(\nabla^{g_2}\psi'_k,\nabla^{g_2}\psi'_l)\Bigr)\\
&=\sum S_1(\nabla^{g_1}b_k,\nabla^{g_1}b_l)\psi'_k\psi'_l+\sum \psi_i\psi_j S_2(\nabla^{g_2}c_i,\nabla^{g_2}c_j).
\end{split}
\end{equation*}
By the integration, we have
\begin{equation}\label{aa}
\begin{split}
&\int_M S(\nabla u,\nabla u)\,d\mu_g\\
=&\sum_k \int_{M_1} S_1(\nabla^{g_1}b_k,\nabla^{g_1}b_k)\,d\mu_{g_1}+\sum_i \int_{M_2} S_2(\nabla^{g_2}c_i,\nabla^{g_2}c_i)\,d\mu_{g_2}\\
\leq& \Lambda_1(S_1)\sum_k \int_{M_1} (\Delta^{g_1}b_k,\Delta^{g_1}b_k)\,d\mu_{g_1}+\Lambda_1(S_2)\sum_i \int_{M_2} (\Delta^{g_2}c_i,\Delta^{g_2}c_i)\,d\mu_{g_2}.
\end{split}
\end{equation}

We have $b_k(x)=\sum a_{ik}\psi_i(x)$, and so 
\begin{equation}\label{bb}
\int_{M_1} (\Delta^{g_1}b_k,\Delta^{g_1}b_k)\,d\mu_{g_1}=\sum_{i} a_{ik}^2 \lambda_i^2.
\end{equation}
Similarly, we have
\begin{equation}\label{cc}
\int_{M_2} (\Delta^{g_2}c_i,\Delta^{g_2}c_i)\,d\mu_{g_2}=\sum_{k} a_{ik}^2 \lambda'^2_k.
\end{equation}

By (\ref{aa})--(\ref{cc}) and $\Delta (\psi_i \psi'_k)=(\lambda_i+\lambda'_k)\psi_i \psi'_k$, we get
\begin{equation*}
\begin{split}
\int_M (\Delta u,\Delta u)\,d\mu_g
&=\sum_{i,k} a_{ik}^2 (\lambda_i+\lambda'_k)^2\\
&\geq \sum_{i,k} a_{ik}^2 \lambda_i^2+\sum_{ik} a_{ik}^2 \lambda'^2_k\\
&=\sum_k \int_{M_1} (\Delta^{g_1}b_k,\Delta^{g_1}b_k)\,d\mu_{g_1}+\sum_i \int_{M_2} (\Delta^{g_2}c_i,\Delta^{g_2}c_i)\,d\mu_{g_2}\\
&\geq \frac{1}{\max\{\Lambda_1(S_1),\Lambda_1(S_2)\}}\Bigl(\Lambda_1(S_1)\sum_k \int_{M_1} (\Delta^{g_1}b_k,\Delta^{g_1}b_k)\,d\mu_{g_1}\\
&\hspace{100pt}+\Lambda_1(S_2)\sum_i \int_{M_2} (\Delta^{g_2}c_i,\Delta^{g_2}c_i)\,d\mu_{g_2}\Bigr)\\
&\geq \frac{1}{\max\{\Lambda_1(S_1),\Lambda_1(S_2)\}} \int_M S(\nabla u,\nabla u)\,d\mu_g.
\end{split}
\end{equation*}

Therefore,
\begin{equation}\label{tisai}
\Lambda_1(S)\leq \max\{\Lambda_1(S_1),\Lambda_1(S_2)\}.
\end{equation}

Next, we prove that the equality holds.
Suppose that $\Lambda_1(S_1)\geq\Lambda_1(S_2)$.
Moreover, we assume $\Lambda_1(S_1)>0$ (otherwise the proposition is trivial).
Take a non-constant function $v_1 \in C^\infty (M_1)$ such that $\Lambda_1(S_1)=\Lambda_{S_1}(v_1)$.
We regard $v_1$ as a smooth function on $M$; $v_1\in C^\infty(M)$. Then,
\begin{equation}\label{dekai}
\Lambda_1(S)\geq \Lambda_S(v_1)= \Lambda_{S_1}(v_1)=\Lambda_1(S_1)=\max\{\Lambda_1(S_1),\Lambda_1(S_2)\}.
\end{equation}
By (\ref{tisai}) and (\ref{dekai}), we get $\Lambda_1(S)=\max\{\Lambda_1(S_1),\Lambda_1(S_2)\}$. 
\end{proof}
We give a useful formula for $\Lambda_i(S)$.
\begin{Prop}\label{mnmx}
Let \((M,g)\) be a closed Riemannian manifold and $S$ a symmetric tensor of type $(0,2)$ on $M$.
We have  
\begin{equation*}
\Lambda_k(S)=\sup_{L_k\subset H} \left[\inf_{v\in L_k\backslash \{0\}}\Lambda_S(v)\right]
\end{equation*}
for any positive integer $k\in\mathbb{Z}_{>0}$,
where $L_k$ varies over all $k$-dimensional subspaces of $H$.
\end{Prop}
\begin{proof}
If $\Lambda_1(S)=0$, we have $S\leq 0$.
Similarly to Lemma \ref{lem}, we can show
\begin{equation*}
\sup_{L_k\subset H} \left[\inf_{v\in L_k\backslash \{0\}}\Lambda_S(v)\right]=0
\end{equation*}
for any $k\in\mathbb{Z}_{>0}$, and so the proposition holds.

Suppose that $\Lambda_1(S)>0$.
Let $\{v_k\}_{\mathbb{Z}>0}$ be the associated functions to $\{\Lambda_k(S)\}_{\mathbb{Z}>0}$.
Take a positive integer $k\in\mathbb{Z}_{>0}$.
Let $L_k$ be a $k$-dimensional subspace of $H$.
Then, there exists a function $v\in L_k \backslash \{0\}$ such that $v$ is orthogonal to $v_1,\ldots,v_{k-1}$.
By definition of $\Lambda_k(S)$, we have $\Lambda_S(v)\leq \Lambda_k(S)$.
Therefore, $\inf_{v\in L_k\backslash \{0\}}\Lambda_S(v)\leq \Lambda_k(S)$ holds.
Thus,
\begin{equation*}
\Lambda_k(S)\geq \sup_{L_k\subset H} \left[\inf_{v\in L_k\backslash \{0\}}\Lambda_S(v)\right].
\end{equation*}

Let us show the equality holds.
We put $L^0_k=\Span_\mathbb{R}\{v_1,\cdots,v_k\}$.
Then, we have $\Lambda_k(S)=\inf_{v\in L^0_k\backslash \{0\}}\Lambda_S(v)$.
This implies
\begin{equation*}
\Lambda_k(S)=\sup_{L_k\subset H} \left[\inf_{v\in L_k\backslash \{0\}}\Lambda_S(v)\right].
\end{equation*}
\end{proof}
\begin{Cor}\label{comp}
Let $S$ and $T$ be symmetric tensors of type $(0,2)$ on $M$.
If $S\geq T$, then we have $\Lambda_k(S)\geq \Lambda_k(T)$ for each $k\in\mathbb{Z}\backslash\{0\}$.
Moreover, if $\Lambda_k(S)= \Lambda_k(T)$ for all $k\in\mathbb{Z}\backslash\{0\}$, then $S=T$ holds.
\end{Cor}
\begin{proof}
Take a positive integer $k\in\mathbb{Z}_{>0}$.
We have
\begin{equation*}
\inf_{v\in L_k\backslash \{0\}}\Lambda_S(v)\geq \inf_{v\in L_k\backslash \{0\}}\Lambda_T(v)
\end{equation*}
for any $k$-dimensional subspace $L_k$ of $H$.
Thus, we get $\Lambda_k(S)\geq \Lambda_k(T)$ by Proposition \ref{mnmx}.
Similarly, we have $\Lambda_k(-T)\geq \Lambda_k(-S)$;
therefore, $\Lambda_{-k}(S)\geq \Lambda_{-k}(T)$.

Suppose that $\Lambda_k(S)= \Lambda_k(T)$ holds for all $k\in \mathbb{Z}\backslash\{0\}$.
If $\Lambda_1(T)=0$, then $H^+(S)=H^+(T)=\{0\}$.
Assume that $\Lambda_1(T)>0$ and take the associated functions $\{v_k\}_{\mathbb{Z}_{>0}}$ to $\{\Lambda_k(T)\}_{\mathbb{Z}_{>0}}$ for $T$.
Take a positive integer $l\in\mathbb{Z}_{>0}$ such that $\Lambda_{l-1}(T)<\Lambda_l(T)$ (put $\Lambda_0(T)=0$).
Suppose that $\Lambda_l(T)=\Lambda_{l+K}(T)$ for $K\in\mathbb{Z}_{\geq0}$.
For each $i=0,\ldots,K$, we put $L^0_{l,i}=\Span_\mathbb{R}\{v_1,\cdots,v_{l-1},v_{l+i}\}$.
By Proposition \ref{mnmx}, we have 
\begin{equation*}
\Lambda_{l+i}(S)=\Lambda_l(S)\geq \inf_{v\in L^0_{l,i}\backslash \{0\}}\Lambda_S(v).
\end{equation*}
Take $v \in L^0_{l,i}\backslash \{0\}$ that attains the infimum of the right hand side.
Then, we have $\Lambda_{l+i}(S)\geq \Lambda_S(v)\geq \Lambda_T(v)\geq \Lambda_{l+i}(T)$.
By $\Lambda_{l+i}(S)=\Lambda_{l+i}(T)$, we have $v=v_{l+i}$ and $\Lambda_{l+i}(S)=\Lambda_S(v_{l+i})$.
Thus, $\{v_k\}_{\mathbb{Z}_{>0}}$ is the family of associated functions to $\{\Lambda_k(S)\}_{\mathbb{Z}_{>0}}$ for $S$, and so $\int_M S(\nabla v,\nabla v)\,d\mu_g = \int_M T(\nabla v,\nabla v)\,d\mu_g$ for any $v\in H^+(S)=H^+(T)$.
Similarly, we have $H^-(S)=H^-(T)$ and $\int_M S(\nabla v,\nabla v)\,d\mu_g = \int_M T(\nabla v,\nabla v)\,d\mu_g$ for any $v\in H^-(S)=H^-(T)$.
Consequently, we have $H^0(S)=H^0(T)$ and
 $\int_M S(\nabla v,\nabla v)\,d\mu_g = \int_M T(\nabla v,\nabla v)\,d\mu_g$ for any $v\in H$.
This and Lemma \ref{one} imply $S=T$.  
\end{proof}
\section{Main properties of $\Omega_k$}
\subsection{The proofs of Theorem A and Theorem C}
In this section we consider the case when $S=\Ric_g$.
\begin{Def}\label{def2}
We put $\Omega_k(g)=\Lambda_k(\Ric_g)$ for a closed Riemannian manifold $(M,g)$.
\end{Def}
In this case, the following Bochner formula plays an important role.
\begin{Prop}
Let $(M,g)$ be a Riemannian manifold and $\omega$ a $1$-form on $M$.
Then, we have
\begin{equation*}
\Delta \omega =\nabla^\ast \nabla \omega + \Ric(\omega^\sharp,\cdot),
\end{equation*}
where $\omega^\sharp$ denotes a vector field such that $g(\omega^\sharp,Y)=\omega(Y)$ for any vector field $Y$.
\end{Prop}
By the Bochner formula, we have the following theorem.
\begin{A}\label{main1}
For any closed $n$-dimensional Riemannian manifold $(M,g)$, we have $\Omega_1(g) \leq \frac{n-1}{n}$. 
If $\Omega_1(g)=\frac{n-1}{n}$, then $(M,g)$ is conformal to the standard sphere.
\end{A}
\begin{proof}
For any function $v\in C^\infty(M)$, we have $\Delta v=-\tr\nabla^2 v$; therefore,
\begin{equation*}
0\leq (\nabla^2 v + \frac{1}{n} \Delta v g,\nabla^2 v + \frac{1}{n} \Delta v g)=(\nabla^2 v,\nabla^2 v)-\frac{1}{n}(\Delta v,\Delta v).
\end{equation*}
Thus, we have
\begin{equation}\label{fundlich}
\|\Delta v\|_{L^2}^2 \leq n \|\nabla^2 v\|_{L^2}^2.
\end{equation}
By the Bochner formula and (\ref{fundlich}), we have
\begin{equation*}
\begin{split}
\int_M \Ric(\nabla v,\nabla v)\,d\mu_g
&=\int_M (\Delta dv,dv)\,d\mu_g-\int_M (\nabla^\ast \nabla dv,dv)\,d\mu_g\\
&=\int_M (d\Delta v,dv)\,d\mu_g-\int_M (\nabla^2 v,\nabla^2 v)\,d\mu_g\\
&=\int_M (\Delta v,\Delta v)\,d\mu_g-\int_M (\nabla^2 v,\nabla^2 v)\,d\mu_g\\
&\leq \frac{n-1}{n} \int_M (\Delta v,\Delta v)\,d\mu_g
\end{split}
\end{equation*}
Consequently, we get $\Lambda_{\Ric}(v)\leq \frac{n-1}{n}$ for any non-constant function $v\in C^\infty(M)$.
Therefore, we get $\Omega_1(g) \leq \frac{n-1}{n}$.

Suppose that $\Omega_1(g)= \frac{n-1}{n}$.
Then, there exists a function $v_1 \in H$ such that  $\Lambda_{\Ric}(v_1)=\frac{n-1}{n} $, and so  $\|\Delta v_1\|_{L^2}^2 = n \|\nabla^2 v_1\|_{L^2}^2$.
Therefore, we have $\nabla^2 v_1 + \frac{1}{n} \Delta v_1 g=0$.
Thus, Theorem \ref{Tas} implies the theorem.
\end{proof}
\begin{Rem}
We have
\begin{equation*}
\|\Delta v\|_{L^2}^2 \leq \frac{1}{1-\Omega_1(g)} \|\nabla^2 v\|_{L^2}^2
\end{equation*}
for any $v\in C^\infty(M)$.
\end{Rem}

\begin{Prop}
Let $0<\lambda_1 \leq \lambda_2 \leq \cdots \to \infty$ be the eigenvalues of the Laplacian.
If there exists a constant $r>0$ such that $\Ric \geq r g$, then we have
\begin{equation*}
\lambda_k \geq \frac{r}{\Omega_k(g)}
\end{equation*}
for any $k=1,2,\ldots$.
Moreover, if $\lambda_k = \frac{r}{\Omega_k(g)}$ holds for all $k=1,2,\ldots$, then $\Ric= r g$ holds.
\end{Prop}
\begin{proof}
By Corollary \ref{koyuti2}, we have $\Lambda_k(r g)=\frac{r}{\lambda_k}$.
Thus, Corollary \ref{comp} implies the proposition.
\end{proof}

Putting $k=1$, we get the following.
\begin{C}\label{Lich}
If there exists a constant $r>0$ such that $\Ric \geq r g$, then
\begin{equation*}
\lambda_1 \geq\frac{r}{\Omega_1(g)}\geq \frac{n}{n-1}r.
\end{equation*}
\end{C}

Let us consider the behavior of $\Omega_k(g)$ under the homothetic transformation.
Let $(M,g)$ be a closed Riemannian manifold and  $a$ a positive constant.
Then, for any $u\in C^\infty(M)$ we have $\nabla^{a^2g}u=\frac{1}{a^2} \nabla^g u$, $\Delta^{a^2g} u=\frac{1}{a^2}\Delta^g u$, $\Ric_{a^2 g}=\Ric_{g}$ and $d\mu_{a^2g}= a^n d\mu_g$.
Thus, if $u$ is a non-constant function, we have
\begin{equation*}
\frac{\int_M \Ric_{a^2g}(\nabla^{a^2g} u,\nabla^{a^2g}u)\,d\mu_{a^2 g}}{\int_M (\Delta^{a^2g} u,\Delta^{a^2g}u)\,d\mu_{a^2g}}=\frac{\int_M \Ric_{g}(\nabla^{g} u,\nabla^{g}u)\,d\mu_{g}}{\int_M (\Delta^{g} u,\Delta^{g}u)\,d\mu_{g}}.
\end{equation*}
Therefore, we have $\Omega_k(a^2 g)=\Omega_k(g)$ for any $k\in \mathbb{Z}\backslash\{0\}$.

On the other hand, $\lambda_k(a^2g)=\frac{1}{a^2} \lambda_k(g)$ holds.
Based on these, we give an example of non-Einstein Riemannian manifold such that $\Ric \geq r g$ and $\lambda_1 =\frac{r}{\Omega_1(g)}$ for some positive constant $r>0$.
\begin{Example}
Let $(M_i,g_i)$ $(i=1,2)$ be closed Einstein manifolds of dimension $n_i$ such that $\Ric_{g_i}=g_i$ holds.
We assume that $\Omega_1(g_1)\geq\Omega_1(g_2)$.
For each $r>1$, we consider the metric $G_r=r^2 g_1 + g_2$ on the product manifold $M=M_1 \times M_2$.
If $r$ is large enough, we have $\lambda_1(G_r)=\min \{\frac{1}{r^2} \lambda_1(g_1),\lambda_1(g_2)\}=\frac{1}{r^2} \lambda_1(g_1)$.

We have
\begin{equation*}
\Ric_{G_r}=\Ric_{g_1}\oplus \Ric_{g_2}=g_1 + g_2 \geq \frac{1}{r^2}G_r,
\end{equation*}
and $\Omega_1(G_r)=\Omega_1(g_1)$ by Proposition \ref{prodmf}.
Therefore, we get
\begin{equation*}
\frac{1}{r^2} \frac{1}{\Omega_1(g_1)}=\frac{1}{r^2} \lambda_1(g_1)=\lambda_1(G_r).
\end{equation*}
This gives an example mentioned above. 
\end{Example}
\subsection{The proof of Theorem B}
In this subsection, we prove Theorem B.
\begin{B}\label{chr}
Let $(M,g)$ be a closed Riemannian manifold of dimension $n$.
If $\Omega_1(g)=\Omega_2(g)=\frac{n-1}{n}$, then $(M,g)$ is isometric to the standard sphere of a certain radius.
\end{B}

The proof of this theorem depends on the following lemma.
For a related work, see \cite[Proposition 7.3]{T2}.
\begin{Lem}\label{clam}
Let $(M,g)$ be a closed Riemannian manifold of dimension $n$.
If there are linearly independent functions $u_1,u_2\in C^\infty(M) \cap H$ such that $\nabla^2 u_i=-\frac{\Delta u_i}{n}g\,(i=1,2)$, then $(M,g)$ is isometric to the standard sphere of a certain radius.
\end{Lem}

Before giving a proof of Lemma \ref{clam}, we recall Tashiro's work \cite{T1}.
First we enumerate the properties of the equation  $\nabla^2 u=-\frac{\Delta u}{n}g$.
\begin{Cha}
If there exists a non-constant function $u\in C^\infty(M)$ such that $\nabla^2 u=-\frac{\Delta u}{n}g$, then $(M,g)$ and $u$ have the following properties:
\begin{itemize}
\item[(i)] $u$ has just two critical points $\{p,q\}$ such that $u(p)<u(q)$.
\item[(ii)] There exist a constant $b>0$, a smooth function $\psi \colon [0,b]\to \mathbb{R}$ and a diffeomorphism $\alpha \colon M\backslash\{p,q\}\to(0,b)\times S^{n-1}$ such that
$u(\alpha^{-1}(t,x))=\psi(t)\,((t,x)\in (0,b)\times S^{n-1})$ holds,
\item[(iii)] For any $t\in(0,b)$, we have $\psi'(t)>0$. For any positive integer $k\in \mathbb{Z}_{>0}$, $\psi^{(2k-1)}(0)=\psi^{(2k-1)}(b)=0$ holds. Moreover, $\psi''(0)=-\psi''(b)>0$ holds,
\item[(iv)] Under the diffeomorphism $\alpha$, the metric $g$ on $M\backslash\{p,q\}$ is represented as $g(t,x)=dt^2 + \frac{{\psi'(t)}^2}{{\psi''(0)}^2}g_{n-1}(x)$, where $g_{n-1}$ denotes the standard metric on the sphere of radius $1$,
\item[(v)] Define $P_1 \colon M\backslash\{q\}\to B^n(b)$ and $P_2\colon M\backslash\{p\}\to B^n(b)$ by $P_1(\alpha^{-1}(t,x))=t x$, $P_1(p)=0$ and $P_2(\alpha^{-1}(t,x))=(b-t) x$, $P_2(q)=0$ respectively, where $B^n(b)$ denotes the open ball of radius $b$ in $n$-dimensional Euclidean space with center $0$.
Then, $P_1$ (resp. $P_2$) is the geodesic coordinate centered at $p$ (resp. $q$).
\end{itemize}
\end{Cha}
In particular, if there exists a non-constant function $u\in C^\infty(M)$ such that $\nabla^2 u=-\frac{\Delta u}{n}g$, then $(M,g)$ is diffeomorphic to $S^n$ and has rotational symmetry.
Let us explain why $(M,g)$ is conformal to the standard sphere.
Fix $t_0 \in (0,b)$ and put a map $\theta\colon (0,b)\to(0,\pi)$ to be $\theta(t)=2 \arctan \exp \int^t_{t_0} \frac{\psi''(0)}{\psi'(s)}\,ds$.
By (iii), $\psi'(s)=O(s)$ near $s=0$, and so $\lim_{t\to 0}\theta(t)=0$.
Similarly, we have $\lim_{t\to b}\theta(t)=\pi$; therefore $\theta$ is a diffeomorphism.
We have 
$\frac{d\theta}{dt}=\sin \theta \frac{\psi''(0)}{\psi'(t)}$, and so $g(t,x)=\frac{1}{\sin^2\theta}\frac{\psi'(t)^2}{\psi''(0)^2}(d\theta^2 + \sin^2 \theta\, g_{n-1})$.
We know that $d\theta^2 + \sin^2 \theta \, g_{n-1}$ is the standard metric on the sphere of radius $1$.
Moreover, $\frac{1}{\sin^2\theta}\frac{\psi'(t)^2}{\psi''(0)^2}$ can be extended to a smooth function on the whole $M$.
This implies that $(M,g)$ is conformal to the standard sphere.

By the formulas for a warped product \cite[Lemma 7.3]{bo} \cite[Lemma 13]{Kh}, we have
\begin{formula}
Let $X$ and $Y$ be vector fields on $M\backslash\{p,q\}$ that are orthogonal to $\frac{\partial}{\partial t}$.
\begin{itemize}
\item[(a)] $\nabla_X Y(t,x)=-g(X,Y)\frac{\psi''(t)}{\psi'(t)}\frac{\partial}{\partial t}+\nabla^{S^{n-1}}_X Y$,
\item[(b)] $R^M(X,\frac{\partial}{\partial t})\frac{\partial}{\partial t}=-\frac{\psi'''(t)}{\psi'(t)}X$, 
\item[(c)] $\Ric(\frac{\partial}{\partial t},\frac{\partial}{\partial t})(t)=-(n-1)\frac{\psi'''(t)}{\psi'(t)}$,
\end{itemize}
where $R^M$ denotes the Riemann curvature tensor on $(M,g)$.
\end{formula}
By (a), if $t_1\in(0,b)$ satisfies $\psi''(t_1)=0$, then the second fundamental form of the embedding $\{t_1\}\times S^{n-1}\subset M$ is equal to $0$.
(b) immediately implies (c).
By (c), $(M,g)$ has the following property:
\begin{Ch}
For any normal geodesics $\gamma_1(t)$ and $\gamma_2(t)$ from $p\in M$, we have 
$\Ric(\dot{\gamma_1}(t),\dot{\gamma_1}(t))=\Ric(\dot{\gamma_2}(t),\dot{\gamma_2}(t))$.
\end{Ch}

Now, we are in position to prove Lemma \ref{clam}.
\begin{proof}[Proof of Lemma \ref{clam}]
We use the notations of Property A for $u_1$: critical points $\{p,q\}$, a constant $b>0$, a diffeomorphism $\alpha \colon M\backslash\{p,q\}=(0,b)\times S^{n-1}$, and a map $\psi_1\colon [0,b]\to\mathbb{R}$ such that $u_1(\alpha^{-1}(t,x))=\psi_1(t)\, ((t,x)\in (0,b)\times S^{n-1})$.
By Property A (iii), we have $\psi'_1(0)=\psi'_1(b)=0$, and so there exists $t_1\in(0,b)$ such that $\psi''_1(t_1)=0$.

Let $p_0$ be one of critical points of $u_2$.
Then, there exists a smooth map $\psi_2\colon [0,b]\to \mathbb{R}$ such that $u_2(y)=\psi_2(d(p_0,y))$.
We extend $\psi_2\colon [-b,2b]\to \mathbb{R}$ by 
\begin{equation*}
\psi_2(-t)=\psi_2(t),\psi_2(b+t)=\psi_2(b-t)\quad(t\in[0,b]).
\end{equation*}

Let us show $p\neq p_0$ by contradiction.
Suppose that $p=p_0$.
Then, by Property A (iv), we have $\frac{\psi'_1(t)^2}{\psi''_1(0)^2}=\frac{\psi'_2(t)^2}{\psi''_2(0)^2}$.
Therefore, there exists a constant $C_1\in \mathbb{R}$ such that $\psi'_1(t)=C_1\psi'_2(t)$.
Thus, there exists a constant $C_2\in \mathbb{R}$ such that $\psi_1(t)=C_1\psi_2(t)+C_2$.
Then, we have $C_2=u_1-C_1u_2\in H$, and so $C_2=0$. This is contradiction to the assumption of the linearly independence of $u_1$ and $u_2$.
Therefore, we have $p\neq p_0$. Similarly, we have $q\neq p_0$.

We put $p_0=(t_0,x_0)\in M\backslash\{p,q\}$.
We construct a non-constant function $u_3$ such that $\nabla^2 u_3=-\frac{\Delta u_3}{n}g$ and $\alpha^{-1}(t_1,x_0)$ is one of critical points of $u_3$.
If $t_0=t_1$, we put $u_3=u_2$.
If $t_0\neq t_1$, we put $u_3=\psi'_2(t_1-t_0)u_1-\psi'(t_1)u_2$.
Since $\psi'_1(t_1)\neq0$, we have $u_3\neq 0$.
We have
\begin{equation*}
u_3(\alpha^{-1}(t,x_0))=\psi'_2(t_1-t_0)\psi_1(t)-\psi'(t_1)\psi_2(t-t_0),
\end{equation*}
and so $\frac{\partial}{\partial t}u_3(\alpha^{-1}(t_1,x_0))=0$.
For any tangent vector $X\in T_{\alpha^{-1}(t_1,x_0)}M$ that is orthogonal to $\left(\frac{\partial}{\partial t}\right)_{(t_1,x_0)}$,
we have $Xu_1(\alpha^{-1}(t_1,x_0))=Xu_2(\alpha^{-1}(t_1,x_0))=0$, and so $Xu_3(\alpha^{-1}(t_1,x_0))=0$.
Therefore, $\alpha^{-1}(t_1,x_0)$ is one of critical points of $u_3$.
In both cases, $u_3$ has desired property.

Since $\psi''_1(t_1)=0$, the second fundamental form of the embedding $\{t_1\}\times S^{n-1} \subset M$ is equal to $0$.
Therefore, for any normal geodesic $\gamma(s)$ in $S^{n-1}(1)$ from $x_0\in S^{n-1}$, $\tilde{\gamma}(s)=\alpha^{-1}(t_1,\gamma(\frac{\psi''_1(0)}{\psi'_1(t_1)}s))\in M$ is a normal geodesic in $M$ from $\alpha^{-1}(t_1,x_0)$.
By the symmetry, there exists a constant $C_3\in \mathbb{R}$ such that
\begin{equation*}
\Ric(\dot{\tilde{\gamma}}(s),\dot{\tilde{\gamma}}(s))=C_3
\end{equation*}
holds.
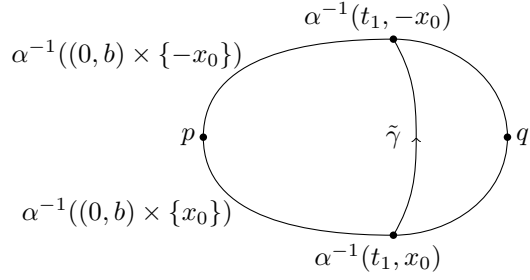
\begin{center}
\begin{tikzpicture}
\fill (0,0) circle (1.4pt);
\fill (2.5,1.3) circle (1.4pt);
\fill (4,0) circle (1.4pt);
\fill (2.5,-1.3) circle (1.4pt);
\draw (0,0)
to [out=90,in=180] (2.5,1.3)
to [out=0,in=90](4,0)
to [out=270,in=0](2.5,-1.3)
to [out=180,in=270](0,0);
\draw[->] (2.5,-1.3)
to [out=60,in=270] (2.8,0);
\draw (2.8,0)
to [out=90,in=300] (2.5,1.3);
\node at (-0.2,0){$p$};
\node at (4.2,0){$q$};
\node at (2.3,1.6){$\alpha^{-1}(t_1,-x_0)$};
\node at (2.3,-1.6){$\alpha^{-1}(t_1,x_0)$};
\node at (-1,-1){$\alpha^{-1}((0,b)\times\{x_0\})$};
\node at (-1,1.1){$\alpha^{-1}((0,b)\times\{-x_0\})$};
\node at (2.5,0){$\tilde{\gamma}$};
\end{tikzpicture}
\end{center}
By Property B for $u_3$, we have 
\begin{equation}\label{const}
-(n-1)\frac{\psi'''_1}{\psi'_1}=\Ric\left(\frac{\partial}{\partial t},\frac{\partial}{\partial t}\right)=\Ric(\dot{\tilde{\gamma}}(s),\dot{\tilde{\gamma}}(s))=C_3.
\end{equation}

By (\ref{const}) and Property A (iii), there exists a constant $C_4$ such that $\psi'_1(t)=C_4 \sin (\frac{\pi}{b}t)$.
Therefore, the metric $g$ on $M$ is represented as 
\begin{equation*}
g(t,x)=dt^2 +\frac{b^2}{\pi^2}\sin^2(\frac{\pi}{b}t)g_{n-1}.
\end{equation*}
Putting $\theta=\frac{\pi}{b}t$, we get $g=\frac{b^2}{\pi^2}(d\theta^2+ \sin^2 \theta \,g_{n-1})$.
This is the standard metric on the sphere of radius $\frac{b}{\pi}$.
\end{proof}
\begin{proof}[Proof of Theorem B]
Lemma \ref{clam} immediately implies the theorem.
\end{proof}

\subsection{An estimate in the presence of a parallel $p$-form}
The goal of this subsection is to prove the following theorem.
\begin{Thm}\label{pfom}
Let \((M,g)\) be a closed Riemannian manifold of dimension $n=2p$ ($p\geq 2$). 
If there exists a nontrivial parallel $p$-form $\omega$ on $M$,
then we have $\Omega_1(g)\leq \frac{p-1}{p}$.
\end{Thm}
The proof of this theorem depends on the method used in \cite{gr}.
Before giving the proof, we recall some definitions and properties about $p$-form.
Let $\omega$ be a $p$-form on $M$ and $\{e_i\}_{1\leq i\leq n}$ a local orthonormal frame.
We write $\omega_{i_1\cdots i_p}$ instead of $\omega(e_{i_1},\cdots,e_{i_p})$.
We define the inner product of $p$-forms $\omega$ and $\theta$ by
\begin{equation*}
(\omega,\theta)=\frac{1}{p!}\sum_{1\leq i_1<\ldots<i_p\leq n} \omega_{i_1\cdots i_p} \theta_{i_1\cdots i_p}.
\end{equation*}
The inner product of a $p$-form $\omega$ with the vector field $X$ is the ($p-1$)-form $\iota(X)\omega$ defined by
\begin{equation*}
(\iota(X)\omega)(Y_1,\cdots,Y_{p-1})=\omega(X,Y_1,\cdots,Y_{p-1}),
\end{equation*}
for any vector field $Y_1,\cdots,Y_{p-1}$.
For any vector field $X$, we define a $1$-form $X^\ast$ by
\begin{equation*}
X^\ast(Y)=g(X,Y)
\end{equation*}
for any vector field $Y$.

Suppose that $M$ is orientable.
Then, for any vector field $X$, $p$-form $\omega$ and ($p-1$)-form $\eta$, we have
\begin{align*}
(\ast(\omega \wedge X^\ast),\eta)d\mu_g=\omega \wedge X^\ast \wedge \eta\\
(\iota(X)\ast \omega,\eta)d\mu_g=(\ast \omega,X^\ast\wedge \eta)d\mu_g=\omega\wedge X^\ast \wedge \eta,
\end{align*}
where $\ast$ denotes the Hodge star operator and $d\mu_g$ denotes the volume form on $(M,g)$.
Thus, we have $\ast(\omega \wedge X^\ast)=\iota(X)\ast \omega$.
Therefore, for any vector field $X$, $Y$ and a $p$-form $\omega$, we have
\begin{equation}\label{hod}
\begin{split}
&(\iota (X)\omega,\iota(Y)\omega)\\
=&(\omega,X^\ast \wedge \iota(Y)\omega)\\
=&(\omega,-\iota(Y)(X^\ast \wedge \omega)+(X,Y)\omega)=-(Y^\ast \wedge \omega,X^\ast \wedge \omega)+(X,Y)(\omega,\omega)\\
=&-(\iota(X)\ast \omega,\iota(Y)\ast \omega)+(X,Y)(\omega,\omega)
\end{split}
\end{equation}

For any function $u\in C^\infty(M)$ and $p$-form $\omega$, we define $p$-tensor $D_{\omega}u$ by
\begin{equation*}
\begin{split}
D_{\omega}u=
&\frac{1}{(p-2)!}\sum_{i,i_1,\cdots,i_{p-2}}\Bigl(\iota(e_i)\iota(e_{i_1})\cdots\iota(e_{i_{p-2}})\omega\otimes\nabla_{e_i} du\\
&+\nabla_{e_i}du \otimes \iota(e_i)\iota(e_{i_1})\cdots\iota(e_{i_{p-2}})\omega \Bigr)\otimes (e^{i_1} \wedge\cdots\wedge e^{i_{p-2}}).
\end{split}
\end{equation*}
Put
\begin{equation*}
\begin{split}
(D_{\omega} u)_{jki_1,\cdots,i_{p-2}}&=D_{\omega} u(e_j,e_k,e_{i_1},\cdots,e_{i_{p-2}}),\\
|D_{\omega} u|^2&=\frac{1}{(p-2)!}\sum_{jki_1\cdots i_{p-2}} (D_{\omega} u)^2_{jki_1,\cdots,i_{p-2}}.
\end{split}
\end{equation*}
If $\omega$ is a parallel $p$-form, then by \cite[Proposition 3.1]{gr}, we have
\begin{equation}\label{gro}
\begin{split}
&\frac{1}{p}\int_M |D_{\omega}u|^2 \,d\mu_g\\
=&\frac{p-1}{p}\int_M (\iota(\nabla u)\omega,\iota(\nabla \Delta u)\omega)\,d\mu_g
-\int_M (\iota(\nabla u)\omega,\iota\left(\Ric(\nabla u,\cdot)\right)\omega)\,d\mu_g.
\end{split}
\end{equation}
\begin{proof}[Proof of Theorem \ref{pfom}]
If $M$ is not orientable, we take an orientable double covering $P \colon \widetilde{M}\to M$ and consider the metric $\tilde{g}=P^\ast g$ on $\widetilde{M}$. Suppose $\Omega_1(g)>0$ and take a function $v\in H_{(M,g)}$ such that $\Delta^2 v=\frac{1}{\Omega_1(g)}\nabla^\ast (\Ric_g(\nabla v))$.
Then, we have $\Delta^2 (v\circ P)=\frac{1}{\Omega_1(g)}\nabla^\ast (\Ric_{\tilde{g}}(\nabla (v\circ P)))$.
Therefore, $\Omega_1(g)\leq \Omega_1(\tilde{g})$ holds.
If $\Omega_1(g)=0$, we obviously have $\Omega_1(g)\leq \Omega_1(\tilde{g})$.
Thus, it is suffice to consider the case when $M$ is orientable.

Suppose that $M$ is orientable.
Take a function $u\in C^\infty(M)$.
Since $\omega$ and $\ast \omega$ is parallel, by (\ref{gro}) we have
\begin{equation}\label{gro2}
\begin{split}
0
&\leq\frac{1}{p}\int_M |D_{\omega}u|^2 \,d\mu_g + \frac{1}{p}\int_M |D_{\ast \omega}u|^2 \,d\mu_g\\
&=\frac{p-1}{p}\int_M (\iota(\nabla u)\omega,\iota(\nabla \Delta u)\omega)\,d\mu_g+\frac{p-1}{p}\int_M (\iota(\nabla u)\ast \omega,\iota(\nabla \Delta u)\ast \omega)\,d\mu_g\\
&\quad -\int_M (\iota(\nabla u)\omega,\iota(\Ric(\nabla u,\cdot))\omega)\,d\mu_g-\int_M (\iota(\nabla u)\ast\omega,\iota(\Ric(\nabla u,\cdot))\ast\omega)\,d\mu_g.
\end{split}
\end{equation}
Note that we assumed $n=2p$, and so $\ast \omega$ is a $p$-form.

By (\ref{hod}) and (\ref{gro2}), we obtain
\begin{equation*}
\begin{split}
0
&\leq \frac{p-1}{p}\int_M (\nabla u,\nabla \Delta u)\,d\mu_g|\omega|^2-\int_M \Ric(\nabla u,\nabla u)\,d\mu_g|\omega|^2\\
&=\frac{p-1}{p}\int_M (\Delta u,\Delta u)\,d\mu_g|\omega|^2-\int_M \Ric(\nabla u,\nabla u)\,d\mu_g|\omega|^2.
\end{split}
\end{equation*}

Consequently, for any $v\in H\backslash\{0\}$, we have $\Lambda_{\Ric}(v)\leq\frac{p-1}{p}$.
Therefore, $\Omega_1(g)\leq \frac{p-1}{p}$ holds.
\end{proof}


\subsection{Riemannian manifolds whose value of $\Omega_1$ is close to $\frac{n-1}{n}$}

In this subsection, we give an example of a Riemannian manifold that is far from $S^n$ but whose value of $\Omega_1$ is close to $\frac{n-1}{n}$.
 
Let $(M,g)$ be a $n$-dimensional closed Riemannian manifold with rotational symmetry, i.e., there exists a non-constant function $u\in C^\infty(M)$ such that $\nabla^2 u=-\frac{\Delta u}{n} g$, and so Property A holds.
We use the notations of Property A: critical points $\{p,q\}$, a constant $b>0$, a diffeomorphism $\alpha \colon M\backslash\{p,q\}=(0,b)\times S^{n-1}$, and a map $\psi\colon [0,b]\to\mathbb{R}$ such that $u(\alpha^{-1}(t,x))=\psi(t)\, ((t,x)\in (0,b)\times S^{n-1})$.
We may assume $u(p)=0$.
There exists a positive constant $C_1>0$ such that
\begin{equation*}
|\psi'(t)|\leq C_1,\quad
|\psi''(t)|\leq C_1,\quad
|\psi^{(4)}(t)|\leq C_1.
\end{equation*}
Since $\psi(0)=0$, $\psi'(0)=0$ and $\psi'''(0)=0$, we have 
\begin{equation}\label{epsilon}
|\psi(t)|\leq C_1 t^2,\quad|\psi'(t)|\leq C_1 t,\quad |\psi'''(t)|\leq C_1 t.
\end{equation}

For any $R>0$, we define
\begin{equation*}
B^M(q,R)=\{z \in M : d^M(q,z)<R\}.
\end{equation*}

Take a $n$-dimensional closed Riemannian manifold $(N,g')$ and a point $y_0\in N$.
Let $(U; x^1,\cdots,x^n)$ be a local coordinate centered at $y_0$.
Take a positive real number $L$ such that
\begin{equation*}
\overline{B^n(2L)}\subset U,
\end{equation*}
where $B^n(2L)$ denotes the open ball in $n$-dimensional Euclidean space of radius $2L$.

For any $l\in(0,L)$, define  $\theta_{l} \colon N\to \mathbb{R}$ by
\begin{empheq}[left={\theta_l(y)=\empheqlbrace}]{align*}
&0 &&\quad (y \in \overline{B^n(l)}),\\
0 < &\theta_l(y) \leq 1 && \quad(y\in B^n(2l)\backslash \overline{B^n(l)}),\\
&1 && \quad (y \in N\backslash B^n(2l)).
\end{empheq}

Take a smooth function $\phi \colon \mathbb{R}\to \mathbb{R}$ such that
\begin{empheq}[left={\phi(t)=\empheqlbrace}]{align*}
&0 &&\quad (t < \frac{3}{2}),\\
0 \leq &\phi(t) \leq 1 && \quad(\frac{3}{2}\leq t\leq 2),\\
&1 && \quad (2<t).
\end{empheq}
There exist a positive constant $C_2>0$ such that $\sup_{t\in \mathbb{R}}|\phi'(t)| \leq C_2$.

For any $l\in(0,L)$ and small $\epsilon>0$, there exists a diffeomorphism 
\begin{equation*}
\psi_{(l,\epsilon)} \colon B^M(q,b-\epsilon) \to B^n(2l)\subset N
\end{equation*}
such that $\psi_{(l,\epsilon)}(B^M(q,b-2\epsilon))=B^n(l)$.
Define $\zeta_{\epsilon}, \eta_{\epsilon} \colon M \to \mathbb{R}$ by
\begin{align*}
\zeta_{\epsilon}(z)&=\phi\left(\frac{d^M(z,p)}{\epsilon}\right),\\
\eta_{\epsilon}(z)&=\phi\left(\frac{d^M(z,p)}{2\epsilon}\right).
\end{align*}

For any $r>0$, $l\in (0,L)$ and small $\epsilon>0$, we define a metric on $N$ by
\begin{equation*}
g_{(r,l,\epsilon)}=(\psi_{(l,\epsilon)}^{-1})^\ast(r\zeta_{\epsilon} g)+ \theta_l g'.
\end{equation*}
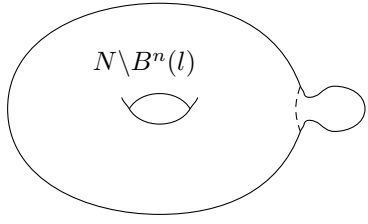
\begin{center}
\begin{tikzpicture}
\draw (0,0)
to [out=90,in=180] (2,1.4)
to [out=0,in=110](3.85,0.3)
to [out=290,in=100](3.9,0.2)
to [out=280,in=220](4.1,0.2)
to [out=40,in=180](4.3,0.3)
to [out=0,in=90](4.7,0)
to [out=270,in=0](4.3,-0.3)
to [out=180,in=320](4.1,-0.2)
to [out=140,in=80](3.9,-0.2)
to [out=260,in=70](3.85,-0.3)
to [out=250,in=0](2,-1.4)
to [out=180,in=270](0,0);
\draw[densely dashed] (3.85,0.3)
to [out=250,in=110](3.85,-0.3);
\draw (1.6,0)
to [out=60,in=120](2.4,0);
\draw (1.5,0.15)
to [out=290,in=120](1.6,0)
to [out=300,in=240](2.4,0)
to [out=60,in=250](2.5,0.15);
\node at (1.8,0.6){$N\backslash{B^n(l)}$};
\end{tikzpicture}
\end{center}
We put $u_{(l,\epsilon)}=(\psi_{(l,\epsilon)}^{-1})^\ast(\eta_{\epsilon} u)\in C^\infty(N)$.
Then, we have
\begin{equation}\label{lam}
\Lambda_{\Ric_{g_{(r,l,\epsilon)}}}(u_{(l,\epsilon)})=\frac{\int_M \Ric_g (\nabla^g (\eta_{\epsilon} u),\nabla^g (\eta_{\epsilon} u))\,d\mu_g}{\int_M (\Delta^g (\eta_{\epsilon} u))^2 \,d\mu_g}.
\end{equation}

By the formulas of a warped product and (\ref{epsilon}), for some positive constant $C_3>0$, we have
\begin{equation}\label{ue}
\begin{split}
&\left| \int_{B^M(p,4\epsilon)} \Ric_g (\nabla^g (\eta_{\epsilon} u),\nabla^g (\eta_{\epsilon} u))\,d\mu_g\right|\\
=&(n-1)\Vol(S^{n-1}(1))\\
&\qquad\times \quad\left|\int_{3\epsilon}^{4\epsilon} \frac{\psi'''(t)}{\psi'(t)} \left(\frac{1}{2\epsilon}\phi' \left(\frac{t}{2\epsilon} \right)\psi(t)+\phi\left(\frac{t}{2\epsilon}\right)\psi'(t)\right)^2\left(\frac{\psi'(t)}{\psi''(0)}\right)^{n-1} \,dt\right|\\
\leq&C_3\epsilon^{n+2}\to 0 \quad \text{($\epsilon \to 0$)}.
\end{split}
\end{equation}
Similarly, for some positive constant $C_4>0$, we have
\begin{equation}\label{sita}
\left|\int_{B^M(p,4\epsilon)} (\Delta^g (\eta_{\epsilon} u))^2 \,d\mu_g\right|
\leq C_4 \epsilon^n \to 0 \quad \text{($\epsilon \to 0$)}
\end{equation}
Note that $\eta_{\epsilon}u=u$ on $M\backslash B^M(p,4\epsilon)$.
By (\ref{lam}), (\ref{ue}) and (\ref{sita}), we have
\begin{equation*}
\begin{split}
&\frac{n-1}{n}\\
\geq& \Omega_1(g_{(r,l,\epsilon)})\\
\geq& \frac{\int_M \Ric_g (\nabla^g (\eta_{\epsilon} u),\nabla^g (\eta_{\epsilon} u))\,d\mu_g}{\int_M (\Delta^g (\eta_{\epsilon} u))^2 \,d\mu_g}\\
\to &\frac{\int_M \Ric_g (\nabla^g u,\nabla^g u)\,d\mu_g}{\int_M (\Delta^g u)^2 \,d\mu_g}=\frac{n-1}{n},
\end{split}
\end{equation*}
as $\epsilon \to 0$.
Therefore, we have the following theorem.

\begin{Thm}\label{sph}
The above notations are preserved.
For any sequence $\{(r_i,l_i,\epsilon_i)\}_{i\in \mathbb{N}}$ such that $\epsilon_i \to 0$ as $i\to\infty$,
we have $\lim_{i\to \infty}\Omega_1(g_{(r_i,l_i,\epsilon_i)})=\frac{n-1}{n}$.
\end{Thm}
\begin{Rem}
For any sequence $\{(r_i,l_i,\epsilon_i)\}_{i\in \mathbb{N}}$ such that $(r_i,l_i,\epsilon_i) \to (0,0,0)$ as $i\to\infty$, we have $(N,g_{(r_i,l_i,\epsilon_i)})\to(N,g')$ in the Gromov-Hausdorff topology.
However, $\lim_{i\to \infty}\Omega_1(g_{(r_i,l_i,\epsilon_i)})=\frac{n-1}{n}$ is not necessarily $\Omega_1(g')$.
\end{Rem}
\begin{Cor}
For any closed manifold $N$ of dimension $n$, we have
\begin{equation*}
\sup_{g\in \Met(N)}\Omega_1(g)=\frac{n-1}{n},
\end{equation*}
where $\Met(N)$ denotes the set of all Riemannian metrics on $N$.
\end{Cor}

\section{Computation and examples}


\subsection{The product of Einstein manifolds}
In general, it is difficult to calculate $\Omega_k$ for a product Riemannian manifold.
However, we have the following theorem.
\begin{Thm}\label{produ}
Let $(M_i,g_i)$ be closed Einstein manifolds: $\Ric_{g_i}=a_i g_i$.
Consider the product metric $g=g_1 + g_2$ on $M=M_1 \times M_2$.
Assume that $a_1>0$.
Then, we have
\begin{equation*}
\{\Omega_l(g):l\in \mathbb{Z}_{>0}\}=\left\{\frac{(a_1\lambda_i+a_2\lambda'_k)}{(\lambda_i+\lambda'_k)^2}:(i,k)\neq(0,0) \text{ and } \frac{(a_1\lambda_i+a_2\lambda'_k)}{(\lambda_i+\lambda'_k)^2}>0\right\},
\end{equation*}
where $\{\lambda_i\}$ (resp. $\{\lambda'_k\}$) are the eigenvalues of the Laplacian of $(M_1,g_1)$ (resp. $(M_2,g_2)$).
\end{Thm}
\begin{proof}
Let $\{\psi_i\}$ (resp. $\{\psi'_k\}$) be the eigenfunctions of the Laplacian of $(M_1,g_1)$ (resp. $(M_2,g_2)$).
Then, $\{\frac{1}{\lambda_i+\lambda'_k} \psi_i \psi'_k\}_{(i,k)\neq(0,0)}$ forms a complete orthonormal system of $H_{(M,g)}$.
Moreover, we have 
\begin{equation*}
\nabla^\ast\left(\Ric_g(\nabla(\psi_i \psi'_k),\cdot)\right)=(a_1\lambda_i+a_2\lambda'_k)\psi_i \psi'_k=\frac{(a_1\lambda_i+a_2\lambda'_k)}{(\lambda_i+\lambda'_k)^2}\Delta^2(\psi_i \psi'_k).
\end{equation*}
Therefore, by Corollary \ref{koyuti1}, we get
\begin{equation*}
\{\Omega_l(g):l\in \mathbb{Z}_{>0}\}=\left\{\frac{(a_1\lambda_i+a_2\lambda'_k)}{(\lambda_i+\lambda'_k)^2}:(i,k)\neq(0,0) \text{ and } \frac{(a_1\lambda_i+a_2\lambda'_k)}{(\lambda_i+\lambda'_k)^2}>0\right\}.
\end{equation*}. 
\end{proof}
\begin{Example}
By using Theorem \ref{produ}, we calculate the value of $\Omega_1$ of the unitary group $U(n)$ with bi-invariant metrics.

We know that $U(n)=(\mathbb{T}\times SU(n))/Z$, where  $\mathbb{T}=\{t\in \mathbb{C}:|t|=1\}=S^1$ and $Z=\{(e^{-2\pi\sqrt{-1} k/n},e^{2\pi\sqrt{-1} k/n})\in \mathbb{T}\times SU(n):k=0,1,\ldots,n-1\}$.
Let $g_0$ be the bi-invariant metric on $SU(n)$ defined by
\begin{equation*}
g_0(X,Y)=-\tr(XY)\quad X,Y\in\mathfrak{su}(n).
\end{equation*}
Then, we have
\begin{equation*}
\Ric_{g_0}=\frac{1}{2}n g_0.
\end{equation*}
Let $\widetilde{G}_r$ be a bi-invariant metric on $\mathbb{T}\times SU(n)$ defined by
\begin{equation*}
\widetilde{G}_r=g_1(r)+ g_0
\end{equation*}
for a positive real number $r\in \mathbb{R}_{>0}$, where $g_1(r)$ is the standard metric on $S^1$ of radius $r$.
Let $G_r$ be a bi-invariant metric on $U(n)$ induced by $\widetilde{G}_r$.
Note that $G_{\sqrt{n}}$ coincides with the metric on $U(n)$ defined by
\begin{equation*}
-\tr(XY) \quad X,Y\in \mathfrak{u}(n).
\end{equation*}

We show that $\Omega_1(G_r)=\max\left\{\frac{\frac{1}{2}n^2(n^2-1)}{\left(n^2+\frac{n}{r^2}-1\right)^2},\frac{1}{4}\right\}$.

To do this, we recall the fact about the eigenvalues of $(SU(n),g_0)$.
We can compute the spectrum of compact Lie groups by a representation theoretic way, see e.g. \cite[Lemma 1.1]{S}.
We only describe some results about $(SU(n),g_0)$.
Let $\{\lambda_k(g_0)\}$ be the eigenvalues of $(SU(n),g_0)$.
Then, $\lambda_1(g_0)=n-\frac{1}{n}$ holds, and the first eigenfunction is not $Z$-invariant on $\mathbb{T}\times SU(n)$.
We can choose the first eigenfunction $\psi_1$ on $(SU(n),g_0)$ such that  $\widetilde{\psi}_1(t,x)=t\psi_1(x)$ ($(t,x)\in \mathbb{T}\times SU(n)$) is $Z$-invariant.
On the other hand, we have
\begin{equation*}
2n=
\min\bigl\{\lambda_k(g_0): k\in\mathbb{Z}_{>0} \text{ and the eigenfunction is $Z$-invariant on }\mathbb{T}\times SU(n)\bigr\}.
\end{equation*}
Let $\widetilde{\psi_2}$ be such an eigenfunction.

By Theorem \ref{produ}, the associated functions (see Definition \ref{def1}) on $\mathbb{T}\times SU(n)$ are of the form
\begin{equation*}
t^k\psi_i(x)\quad(t,x)\in \mathbb{T}\times SU(n)
\end{equation*}
for some integer $k\in\mathbb{Z}$ and some eigenfunction $\psi_i$ on $SU(n)$ such that $\Delta \psi_i=\lambda_i(g_0)\psi_i$.
Then,
\begin{equation*}
\nabla^\ast\left(\Ric_{\widetilde{G}_r}(\nabla(t^k\psi_i),\cdot)\right)=\frac{1}{2}n\lambda_i t^k\psi_i=\frac{\frac{1}{2}n\lambda_i}{(\frac{1}{r^2}k^2+\lambda_i)^2}\Delta^2(t^k\psi_i).
\end{equation*}
Moreover, we have
\begin{equation*}
\Omega_1(G_r)=\max \left\{\Omega_k(\widetilde{G}_r):k\in \mathbb{Z}_{>0}\text{ and the associated function } v_k \text{ is $Z$-invariant} \right\}.
\end{equation*}
Thus, the candidates of $\Omega_1(G_r)$ are $\Lambda_{\Ric_{\widetilde{G}_r}}(\widetilde{\psi}_1)$ and $\Lambda_{\Ric_{\widetilde{G}_r}}(\widetilde{\psi}_2)$.
We have
\begin{align*}
\nabla^\ast\left(\Ric_{\widetilde{G}_r}(\nabla\widetilde{\psi}_1,\cdot)\right)
=&\frac{\frac{1}{2}n(n-\frac{1}{n})}{\left(\frac{1}{r^2}+n-\frac{1}{n}\right)^2}\Delta^2 \widetilde{\psi}_1
=\frac{\frac{1}{2}n^2(n^2-1)}{\left(n^2+\frac{n}{r^2}-1\right)^2} \Delta^2 \widetilde{\psi}_1,\\
\nabla^\ast\left(\Ric_{\widetilde{G}_r}(\nabla\widetilde{\psi}_2,\cdot)\right)
=&\frac{1}{4}\Delta^2 \widetilde{\psi}_2.
\end{align*}
Therefore, we get 
\begin{equation*}
\Omega_1(G_r)=\max\left\{\frac{\frac{1}{2}n^2(n^2-1)}{\left(n^2+\frac{n}{r^2}-1\right)^2},\frac{1}{4}\right\}.
\end{equation*}
\end{Example}


\subsection{The case of Heisenberg manifolds}

In this subsection,  we consider the value of $\Omega_1$ of the Heisenberg manifolds.
We refer to \cite{gw}.
Let $n\in \mathbb{Z}_{>0}$.
For $x,y\in\mathbb{R}^n$ and $t\in\mathbb{R}$, let
\begin{equation}
\gamma(x,y,t)=
\begin{bmatrix}
1 &x &t\\
0 &I_n & {}^t\!y\\
0 &0 &1
\end{bmatrix},\quad
X(x,y,t)=
\begin{bmatrix}
0 &x &t\\
0 &0 & {}^t\!y\\
0 &0 &0
\end{bmatrix},
\end{equation}
where we consider $x,y$ as law vectors.
The Heisenberg group $H_n$ is the $(2n+1)$-dimensional Lie group defined by
\begin{equation*}
H_n=\left\{
\gamma(x,y,t)
\in GL(n+2, \mathbb{R}) : x,y \in \mathbb{R}^n \text{ and } t\in \mathbb{R}
\right\},
\end{equation*}
with Lie algebra
\begin{equation*}
\mathfrak{h}_n=\left\{X(x,y,t)
\in M(n+2, \mathbb{R}) : x,y \in \mathbb{R}^n \text{ and } t\in \mathbb{R}
\right\}.
\end{equation*}
Let $\Gamma$ be a discrete subgroup of $H_n$ such that $M=\Gamma \backslash H_n$ is compact (such a subgroup is called a uniform discrete subgroup).
We call such a manifold $M$ a Heisenberg manifold.
Let $\Le (M)$ be the set of all left invariant metrics on  $M$.
Then, we have the following proposition.
\begin{Prop}\label{heis}
For any Heisenberg manifolds $M=\Gamma \backslash H_n$, we have
\begin{equation*}
\sup_{g\in \Le(M)}\Omega_1(g)=
\begin{cases}
\frac{1}{16}\frac{\sqrt{17}-3}{(\sqrt{17}-1)(\sqrt{17}+3)^2}\quad &n=1\\
\frac{1}{32}\quad &n\geq2.
\end{cases}
\end{equation*}
\end{Prop}
\begin{Rem}
We have
\begin{align*}
\inf\{\lambda_1(g): g\in \Le(M)\text{ and } \vol(M,g)=1\}&=0,\\
\sup\{\lambda_1(g): g\in \Le(M)\text{ and } \vol(M,g)=1\}&=\infty.
\end{align*}
\end{Rem}
Before giving a proof of the proposition, we recall the basic results described in \cite{gw} and some other facts.
\begin{notation}
We put
\begin{equation*}
X_i=X(e_i,0,0),\quad Y_i=X(0,e_i,0),\quad Z=X(0,0,1)
\end{equation*}
for $1\leq i \leq n$, where $\{e_i\}$ denotes the standard basis of $\mathbb{R}^n$.
We put
\begin{align*}
\mathfrak{z}_n&=\{tZ:t\in\mathbb{R}\},\\
\mathfrak{z}_n^\perp&=\Span_{\mathbb{R}}\{X_1,\ldots,X_n,Y_1,\ldots,Y_n\}.
\end{align*}
Note that $\mathfrak{z}_n$ is the center of $\mathfrak{h}_n$, and
\begin{equation*}
[X_i,X_j]=[Y_i,Y_j]=0, \quad [X_i,Y_j]=\delta_{ij}Z
\end{equation*}
for $1\leq i,j \leq n$.
\end{notation}
\begin{notation}
We put
\begin{equation*}
\mathbb{Z}^n_{\dv}=\{r=(r_1,\ldots,r_n)\in\mathbb{Z}^n_{>0}: r_j \text{ divides } r_{j+1} \text{ for } 1\leq j \leq n-1\}.
\end{equation*}
For $r\in\mathbb{Z}^n_{\dv}$, we put
\begin{align*}
r\mathbb{Z}^n&=\{x=(x_1,\ldots,x_n)\in\mathbb{Z}^n: x_i \in r_i \mathbb{Z} \text{ for } 1\leq i \leq n\},\\
\Gamma_r&=\{\gamma(x,y,t)\in H_n:x\in r\mathbb{Z}^n,y\in\mathbb{Z}^n,t\in \mathbb{Z}\},\\
\mathcal{A}_r&=\{\tau\in \mathfrak{h}_n^\ast: \tau(Z)=0 \text{ and }\tau(\log \Gamma_r)\subset \mathbb{Z}\}.
\end{align*}
\end{notation}
By \cite[Lemma 3.5]{gw}, for any uniform discrete subgroup $\Gamma$ of $H_n$, there exist an automorphism $\psi\colon H_n\to H_n$ and $r\in\mathbb{Z}^n_{\dv}$ such that
\begin{equation*}
\psi(\Gamma)=\Gamma_r
\end{equation*}
holds.
Therefore, $\Gamma \backslash H_n$ is identified with $\Gamma_r \backslash H_n$, and so it is suffice to consider the case when $\Gamma=\Gamma_r$.

By the \cite[Remark 2.6]{gw}, for any metric $g$ on $\mathfrak{h}_n=\mathfrak{z}_n^\perp \oplus \mathfrak{z}_n$, there exists an inner automorphism $\varphi\colon H_n\to H_n$ such that $\varphi^\ast g$ is of the form
\begin{equation*}
\varphi^\ast g=
\begin{bmatrix}
h &0 \\
0 &g_{n+1}
\end{bmatrix}
\end{equation*}
with $h$ a metric on $\mathfrak{z}_n^\perp$ and $g_{2n+1}>0$.
By \cite[Proposition 2.2]{gw}, $(\Gamma_r \backslash H_n,g)$ is isometric to $(\Gamma_r \backslash H_n,\varphi^\ast g)$.
Thus, it is suffice to consider the case when $g$ is of the form
\begin{equation*}
g=
\begin{bmatrix}
h &0 \\
0 &g_{n+1}
\end{bmatrix}.
\end{equation*}

We next explain the fact about the irreducible unitary representations of $H_n$.
\begin{notation}
We consider the following irreducible unitary representations of $H_n$. 
\begin{itemize}
\item[(a)] For $\tau\in \mathfrak{h}_n^\ast$ with $\tau(Z)=0$, we put $f_{\tau}\colon H_n\to U(1)$ by $f_{\tau}(\exp X)=\exp(2\pi \sqrt{-1}\tau(X))$ for all $X\in \mathfrak{h}_n$.
Then, $f_{\tau}$ defines a $1$-dimensional unitary representation of $H_n$. 
\item[(b)] For $c\in \mathbb{R}\backslash \{0\}$, we define a representation $\pi_c$ on $L^2(\mathbb{R}^n)$ by
\begin{equation*}
(\pi_c(\gamma(x,y,t))f)(u)=\exp(2\pi\sqrt{-1}c(t+u\cdot y))f(x+u)
\end{equation*}
for all $f\in L^2(\mathbb{R}^n)$ and $\gamma(x,y,t)\in H_n$.
\end{itemize}
\end{notation}
By \cite[Lemma 3.7]{gw}, 
\begin{equation*}
\{f_{\tau}:\tau\in \mathfrak{h}_n^\ast \text{ and }\tau(Z)=0\}\cup\{\pi_c:c\in\mathbb{R}\backslash\{0\}\}
\end{equation*}
is a complete set of irreducible unitary representations of $H_n$.

Let $R$ be the quasi-regular representation of $H_n$ on $L^2(\Gamma_r \backslash H_n)$, i.e.,
\begin{equation*}
(R(\gamma')f)([\gamma])=f([\gamma\gamma'])
\end{equation*}
holds for all $f\in L^2(\Gamma_r \backslash H_n)$, $[\gamma]\in\Gamma_r \backslash H_n$ and $\gamma\in H_n$.
By \cite[Lemma 3.7]{gw}, $(R,L^2(\Gamma_r \backslash H_n))$ decomposes
\begin{align*}
&(R,L^2(\Gamma_r \backslash H_n))\\
\cong&\left(\bigoplus_{\tau\in\mathcal{A}_r}(f_{\tau},\mathbb{C})\right)\bigoplus\left(\bigoplus_{c\in\mathbb{Z}\backslash\{0\}}(|c^n|r_1\ldots r_n)(\pi_c,L^2(\mathbb{R}^n))\right).
\end{align*}

Finally, we remark on the Laplacian and Ricci curvature of a unimodular Lie group with a left invariant metric.
\begin{Lem}\label{linv}
Let $G$ be an $m$-dimensional unimodular Lie group with Lie algebra $\mathfrak{g}$, i.e., Trace(ad(X))=0 for all $X\in\mathfrak{g}$.
Take a left invariant metric $g$ on $G$.
Let $\{U_i\}$ be the $g$-orthonormal basis of $\mathfrak{g}$.
Consider $\{U_i\}$ as left invariant vector fields on $G$.
\begin{itemize}
\item[(i)]  We have $\Delta_g f=-\sum^{m}_{i=1} U_i^2 f$ for $f\in C^\infty(G)$,
\item[(ii)] We have
\begin{align*}\Ric(U_i,U_j)=&-\frac{1}{2}\sum_{k=1}^m g([U_k,U_i],[U_k,U_j])-\frac{1}{2}\sum_{k=1}^m g(U_k,[[U_k,U_i],U_j])\\
&+\frac{1}{4}\sum_{k,l=1}^m g(U_i,[U_k,U_l])g(U_j,[U_k,U_l]).
\end{align*}
\item[(iii)] We have 
\begin{equation*}
\nabla^\ast\Ric(\nabla f,\cdot)=-\sum_{i,j=1}^m\Ric(U_i,U_j)U_i U_j f
\end{equation*}
for $f\in C^\infty(G)$.
\end{itemize}
\end{Lem}
Straight calculation implies the lemma.
See \cite[Corollary 1]{Ur} for the proof of (i).
Note that the Heisenberg groups is unimodular.

Now, we are in position to prove Proposition \ref{heis}.
\begin{proof}[Proof of Proposition \ref{heis}]
We can assume that $\Gamma=\Gamma_r$ and that $g$ is of the form
\begin{equation*}
g=
\begin{bmatrix}
h &0 \\
0 &g_{n+1}
\end{bmatrix}.
\end{equation*}

By \cite[Lemma 3.5]{gw}, there exist an $h$-orthonormal basis $\{X'_1,\cdots,X'_n,Y'_1,\cdots,Y'_n\}$ of $\mathfrak{z}_n^\perp$ and positive constants $d_1,\cdots,d_n>0$ such that
\begin{equation*}
[X'_i,X'_j]=[Y'_i,Y'_j]=0,\quad [X'_i,Y'_j]=\delta_{ij} d_i^2 Z
\end{equation*}
for all $1\leq i,j\leq n$.
We put $Z'=(g_{2n+1})^{-1/2}Z$.
Then, $\{X'_1,\cdots,X'_n,Y'_1,\cdots,Y'_n,Z'\}$ is a $g$-orthonormal basis of $\mathfrak{h}_n$, and $[X'_i,Y'_i]=d_i^2 \sqrt{g_{2n+1}} Z'$ holds.
By Lemma \ref{linv} (ii), we have
\begin{align*}
\Ric(X'_i,X'_j)&=\Ric(Y'_i,Y'_j)=-\frac{1}{2}\delta_{ij} d_i^4 g_{2n+1},\\
\Ric(X'_i,Y'_j)&=\Ric(X'_i,Z')=\Ric(Y'_i,Z')=0,\\
\Ric(Z',Z')&=\frac{1}{2}\sum_{i=1}^n d_i^4 g_{2n+1},
\end{align*}
for all $1\leq i,j \leq n$.
Therefore, by Lemma \ref{linv} (iii), we have
\begin{equation*}
\nabla^\ast(\Ric(\nabla f,\cdot))=\frac{1}{2}\sum_{i=1}^n d_i^4 g_{2n+1} ((X'_i)^2+(Y'_i)^2) f -\frac{1}{2}\sum_{i=1}^n d_i^4 g_{2n+1} (Z')^2 f
\end{equation*}
for all $f\in C^\infty(M)$.

Let $f\in L^2(\Gamma_r \backslash H_n)$ belongs to $(f_{\tau},\mathbb{C})$-component for some $\tau\in\mathcal{A}_r$.
Take real numbers $a_i,b_i\in \mathbb{R}$ such that $\tau = \sum (a_i (X'_i)^\ast+b_i (Y'_i)^\ast)$, where $\{(X'_1)^\ast, \ldots, (X'_n)^\ast,$
$(Y'_1)^\ast, \ldots, (Y'_n)^\ast,(Z')^\ast\}$ denotes the dual basis of  $\{X'_1,\cdots,X'_n,Y'_1,\cdots,Y'_n,Z'\}$.
Then, we have
\begin{align*}
\Delta f &= 4\pi^2\sum_{i=1}^n(a_i^2+b_i^2)f,\\
\nabla^\ast(\Ric(\nabla f,\cdot))&=-\sum_{i=1}^n 2\pi^2 d_i^4 g_{2n+1}(a_i^2+b_i^2)f
\end{align*}
for all $f\in C^\infty(M)$.
Therefore, we have 
\begin{equation}\label{tau}
-\frac{\sum_{i=1}^n 2\pi^2 d_i^4 g_{2n+1}(a_i^2+b_i^2)}{(4\pi^2\sum_{i=1}^n(a_i^2+b_i^2))^2} \Delta^2 f =\nabla^\ast(\Ric(\nabla f,\cdot)).
\end{equation}

We next consider the functions that belongs to $(\pi_c,L^2(\mathbb{R}^n))$-component for some $c\in \mathbb{Z}\backslash\{0\}$.
We consider the functions in $L^2(\mathbb{R}^n)$ as the functions in $L^2(\Gamma_r \backslash H_n)$.
We define $\psi_*\colon \mathfrak{h}_n \to \mathfrak{h}_n$ by $\psi_*(X'_i)=d_i X_i,\psi_*(Y'_i)=d_i Y_i$ and $\psi_*(Z)=Z$.
Then, $\psi_*$ is automorphism of $\mathfrak{h}_n$ and lifts to the automorphism $\psi\colon H_n\to H_n$.
We put $\pi'_c=\pi_c\circ \psi$.
Then, $\pi'_c$ is irreducible representation of $H_n$ on $L^2(\mathbb{R}^n)$, and $\pi'_c(Z)=\pi_c(Z)$.
Thus, $\pi'_c$ is unitary equivalent to $\pi_c$, and so there exists an isomorphism $\Psi \colon L^2(R^n)\to L^2(R^n)$ such that $\Psi(\pi_c(\gamma)f)=\pi'_c(\gamma)\Psi(f)$ for all $f\in L^2(R^n)$ and $\gamma\in H_n$.
We have $(\pi'_c)_*(X'_i)=d_i (\pi_c)_*(X_i)$,$(\pi'_c)_*(Y'_i)=d_i (\pi_c)_*(Y_i)$ and $(\pi'_c)_*(Z')=(g_{2n+1})^{-1/2}(\pi_c)_*(Z)$.
If  $\Psi (f)\in \mathcal{S}(R^n)$ (where $\mathcal{S}(\mathbb{R}^n)$ denotes the Schwartz space), we have
\begin{align*}
&\Psi (\Delta f)(u)\\
=&\Psi \left(-\sum_{i=1}^n\left(((\pi_c)_\ast(X'_i))^2+((\pi_c)_\ast(Y'_i))^2\right)f-((\pi_c)_\ast(Z'))^2f\right)(u)\\
=&\left[-\sum_{i=1}^n(((\pi'_c)_\ast(X'_i))^2+((\pi'_c)_\ast(Y'_i))^2)-((\pi'_c)_\ast(Z'))^2\right]\Psi(f)(u)\\
=&\left[-\sum_{i=1}^n d_i^2(((\pi_c)_\ast(X_i))^2+((\pi_c)_\ast(Y_i))^2)-\frac{1}{g_{2n+1}}((\pi_c)_\ast(Z))^2\right]\Psi(f)(u)\\
=&\left[\sum_{i=1}^n d_i^2(4\pi^2 c^2 u_i^2- \frac{{\partial}^2}{\partial u_i^2})+\frac{4\pi^2 c^2}{g_{2n+1}}\right]\Psi(f)(u),\\
&\Psi(\nabla^\ast (\Ric(\nabla f,\cdot)))\\
=&\left[\frac{1}{2}\sum_{i=1}^n d_i^6 g_{2n+1} (((\pi_c)_\ast(X_i))^2+((\pi_c)_\ast(Y_i))^2)-\frac{1}{2}\sum_{i=1}^n d_i^4 ((\pi_c)_\ast(Z))^2\right] \Psi(f)\\
=&\left[-\frac{1}{2}\sum_{i=1}^n d_i^6 g_{2n+1} (4\pi^2 c^2 u_i^2- \frac{{\partial}^2}{\partial u_i^2})+2\sum_{i=1}^n d_i^4 \pi^2 c^2\right] \Psi(f).
\end{align*} 
For $k=(k_1,\ldots,k_n)\in\mathbb{Z}_{\geq0}^n$, we put
\begin{equation*}
h_k(u)=\exp(|u|^2/2) \partial^k \exp(-|u|^2)\quad (u\in \mathbb{R}^n),
\end{equation*}
where
\begin{equation*}
\partial^k =\frac{\partial^{k_1+\cdots+k_n}}{\partial u_1^{k_1}\cdots\partial u_n^{k_n}}.
\end{equation*}
These functions are known as Hermite functions, and form a complete orthonormal system of $L^2(\mathbb{R}^n)$.
We have
\begin{equation*}
\left(u_i^2-\frac{\partial^2}{\partial u_i^2}\right)h_k=(2k_i+1)h_k
\end{equation*}
for $1\leq i\leq n$.
Thus, for $\widetilde{h}_k(u)= h_k(\sqrt{2\pi |c|}u)$, we have
\begin{equation*}
\left(4\pi^2 c^2 u_i-\frac{{\partial}^2}{\partial u_i^2}\right) \widetilde{h}_k=2\pi|c|(2k_i+1)\widetilde{h}_k.
\end{equation*}
Therefore, for $f_k=\Psi^{-1} (\widetilde{h}_k)$, we have
\begin{align*}
&\Psi (\Delta f_k)(u)\\
=&\left[\frac{4\pi^2 c^2}{g_{2n+1}}+\sum_{i=1}^n 2\pi|c|d_i^2 (2k_i+1)\right]\Psi(f_k)(u),\\
&\Psi(\nabla^\ast (\Ric(\nabla f_k,\cdot)))\\
=&\left[2\pi^2 c^2\sum_{i=1}^n d_i^4 -\sum_{i=1}^n \pi|c| d_i^6 g_{2n+1}(2k_i+1)\right] \Psi(f_k),
\end{align*}
and so
\begin{equation}\label{see}
\frac{2\pi^2 c^2\sum_{i=1}^n d_i^4 -\sum_{i=1}^n \pi|c| d_i^6 g_{2n+1}(2k_i+1)}{\left[\frac{4\pi^2 c^2}{g_{2n+1}}+\sum_{i=1}^n 2\pi|c|d_i^2 (2k_i+1)\right]^2} \Delta^2 f_k=\nabla^\ast (\Ric(\nabla f_k,\cdot)).
\end{equation}
For $c\in \mathbb{Z}\backslash\{0\},k\in\mathbb{Z}^n_{\geq 0}$, we put
\begin{equation*}
\Omega(c,k,g)=\frac{2\pi^2 c^2\sum_{i=1}^n d_i^4 -\sum_{i=1}^n \pi|c| d_i^6 g_{2n+1}(2k_i+1)}{\left[\frac{4\pi^2 c^2}{g_{2n+1}}+\sum_{i=1}^n 2\pi|c|d_i^2 (2k_i+1)\right]^2}.
\end{equation*}

By (\ref{tau}) and (\ref{see}), we have
\begin{equation*}
\Omega_1(g)=\sup\{\Omega(c,k,g):c\in \mathbb{Z}_{>0},k\in\mathbb{Z}^n_{\geq 0}\}.
\end{equation*}
If $\Omega(c,k,g)>0$, then we have $\Omega(c,0,g)\geq\Omega(c,k,g)$.
Thus, we have $\Omega_1(g)=\sup\{\Omega(c,0,g):c\in \mathbb{Z}_{>0}\}.$
We have, for $c\in\mathbb{Z}_{>0}$,
\begin{equation*}
\Omega(c,0,g)=\frac{g_{2n+1}^2}{8\pi^2}\frac{c\sum_{i=1}^n d_i^4 -\frac{g_{2n+1}}{2\pi} \sum_{i=1}^n d_i^6}{c(c+\frac{g_{2n+1}}{2\pi} \sum_{i=1}^n d_i^2)^2}.
\end{equation*}
Put $p=\sum_{i=1}^n d_i^2,q=\sum_{i=1}^n d_i^4$ and $r=\sum_{i=1}^n d_i^6$.
For $x\in \mathbb{R}_{>0}$, 
\begin{equation*}
F(x)=\frac{g_{2n+1}^2}{8\pi^2}\frac{q x-\frac{g_{2n+1}}{2\pi}r}{x(x+\frac{g_{2n+1}}{2\pi}p)^2}
\end{equation*}
takes its maximum at $x=\frac{g_{2n+1}}{2\pi}\frac{r}{q}\left(\sqrt{\frac{9}{16}+\frac{pq}{2r}}+\frac{3}{4}\right)$.
Thus, we have
\begin{align*}
&\Omega(c,0,g)\\
\leq &F\left(\frac{g_{2n+1}}{2\pi}\frac{r}{q}\left(\sqrt{\frac{9}{16}+\frac{pq}{2r}}+\frac{3}{4}\right)\right)\\
=&\frac{1}{2}\frac{q^3}{r^2}\frac{\sqrt{\frac{9}{16}+\frac{pq}{2r}}-\frac{1}{4}}
{(\sqrt{\frac{9}{16}+\frac{pq}{2r}}+\frac{3}{4})(\frac{pq}{r}+\sqrt{\frac{9}{16}+\frac{pq}{2r}}+\frac{3}{4})^2}
\end{align*}
By the Cauchy-Schwartz inequality, we have
\begin{equation*}
q^2=(\sum_{i=1}^n d_i^4)^2\leq (\sum_{i=1}^n d_i^2)(\sum_{i=1}^n d_i^6)=pr,
\end{equation*}
and so
\begin{equation*}
\frac{q^3}{r^2}\leq\frac{pq}{r}.
\end{equation*}
Put $X=\sqrt{\frac{9}{16}+\frac{pq}{2r}}$.
Since $1\leq \frac{pq}{r}\leq n$, we have $\sqrt{\frac{17}{16}}\leq X \leq \sqrt{\frac{9}{16}+\frac{n}{2}}$.
We have $\frac{pq}{r}=2X^2 - \frac{9}{8}$.
Therefore, we get
\begin{equation}\label{sime}
\begin{split}
\Omega(c,0,g)&\leq \frac{(X^2 - \frac{9}{16})(X-\frac{1}{4})}
{(X+\frac{3}{4})(2X^2+X-\frac{3}{8})^2}\\
&=\frac{1}{4} \frac{X-\frac{3}{4}}{(X-\frac{1}{4})(X+\frac{3}{4})^2}.
\end{split}
\end{equation}

If $n\geq2$, the right side of (\ref{sime}) takes its maximum at $X=\frac{5}{4}$,
and so $\Omega(c,0,g)\leq \frac{1}{32}$.
If we take a sequence of left invariant metrics $\{g(l)\}$ such that $d_1(l)=d_2(l)=\text{constant}$, $\lim_{l\to \infty}d_i(l)=0\text{ ($k\geq3$)}$ and
\begin{equation*}
c(l)=\frac{g_{2n+1}(l)}{2\pi}\left(\frac{3r(l)}{4q(l)}+\sqrt{\frac{9}{16}\frac{r(l)^2}{q(l)^2}+\frac{p(l)r(l)}{2q(l)}}\right) \in \mathbb{Z}_{>0},
\end{equation*}
then $\lim_{l\to\infty}\Omega(c(l),0,g(l))=\frac{1}{32}$.
Thus, we get 
\begin{equation*}
\sup_{g\in \Le(M)}\Omega_1(g)=\frac{1}{32}.
\end{equation*}

If $n=1$, we have $X=\sqrt{\frac{17}{16}}$,
and so 
\begin{equation*}\Omega(c,0,g)\leq \frac{\sqrt{\frac{17}{16}}-\frac{3}{4}}{(\sqrt{\frac{17}{16}}-\frac{1}{4})(\sqrt{\frac{17}{16}}+\frac{3}{4})^2}=\frac{1}{16}\frac{\sqrt{17}-3}{(\sqrt{17}-1)(\sqrt{17}+3)^2}.
\end{equation*}
If we take a left invariant metric such that
\begin{equation*}
c=\frac{g_{2n+1}}{2\pi}\left(\frac{3r}{4q}+\sqrt{\frac{9}{16}\frac{r^2}{q^2}+\frac{pr}{2q}}\right)\in \mathbb{Z}_{>0},
\end{equation*}
then we have $\Omega(c,0,g)=\frac{1}{16}\frac{\sqrt{17}-3}{(\sqrt{17}-1)(\sqrt{17}+3)^2}$.
Thus, we get
\begin{equation*}\sup_{g\in \Le(M)}\Omega_1(g)=
\frac{1}{16}\frac{\sqrt{17}-3}{(\sqrt{17}-1)(\sqrt{17}+3)^2}.
\end{equation*}
\end{proof}
\appendix
\section{The proof of Lemma \ref{subsp}}
In this appendix we prove Lemma \ref{subsp}.
\begin{proof}[Proof of Lemma \ref{subsp}]
Suppose that $S\leq0$ does not hold.
Then, there exist $x_0 \in M$ and $X_0 \in T_{x_0} M$ such that $S(X_0,X_0)>0$.
We put $2 \delta = S(X_0,X_0)>0$.
There exist constants $C_1,C_2>0$ and a local coordinate $(U;x^1,x^2,\ldots,x^n)$ centered at $x_0$ such that
\begin{align*}
(\nabla x^1)_{x_0}&=X_0,\\
S(\nabla x^1,\nabla x^1) &\geq \delta,\\
C_1 d\mu_{\mathbb{R}^n} \leq d\mu_g &\leq C_2 d\mu_{\mathbb{R}^n} \quad\text{on } \, U.
\end{align*}
We take a positive integer $N\in \mathbb{N}$ such that $[-\frac{2\pi}{N},\frac{2\pi}{N}]^n \subset U$ and take a smooth function $\psi \colon U \to \mathbb{R}$ such that
\begin{align*}
\supp \psi &\subset \left(-\frac{2\pi}{N},\frac{2\pi}{N}\right)^n,\\
\psi &=1 \quad \text{on }  \,\left[-\frac{\pi}{N},\frac{\pi}{N}\right]^n,\\
0\leq \psi &\leq 1 \quad\text{on }  \,U.
\end{align*}
Since $\psi$ has compact support, we can take a constant $C_3>0$ such that
\begin{align*}
|S(\nabla \psi ,\nabla \psi)|&\leq C_3,\\
|S(\nabla \psi ,\nabla x^1)|&\leq C_3 \quad\text{on }\, U.
\end{align*}
For each positive integer $m\in \mathbb{N}$, we define a function $u_m \colon [-\frac{2\pi}{N},\frac{2\pi}{N}]^n \to \mathbb{R}$ by
\begin{equation*}
u_m(x)=\sin (mNx^1).
\end{equation*}
We regard $\psi u_m$ as a smooth function on $M$: $\psi u_m \in C^\infty(M)$.
For each $k,K\in \mathbb{N}$ we define a $k$-dimensional subspace $V_{k,K}\subset  C^\infty(M)$ by
\begin{equation*}
V_{k,K}=\Span_\mathbb{R} \{ \psi u_{K+1},\cdots ,\psi u_{K+k} \}.
\end{equation*}
Take $a=(a^{K+1},\cdots,a^{k+K}) \in \mathbb{R}^K$ and define $|a|^2=(a^{K+1})^2+\cdots+(a^{K+k})^2$.
In U, we have

\begin{equation}\label{pt}
\begin{split}
&S\left(\nabla\left(\sum_{i=K+1}^{K+k} \psi a^i u_i\right),\nabla\left(\sum_{j=K+1}^{K+k}\psi a^j u_j\right)\right)\\
=&\sum_{i,j=K+1}^{K+k} a^i a^j \Bigl\{\sin(iNx^1) \sin(jNx^1) S(\nabla \psi,\nabla \psi) \\
&\qquad+\left(jN\psi \sin(iNx^1) \cos(jNx^1)+iN\psi \sin(jNx^1) \cos(iNx^1)\right)S(\nabla \psi,\nabla x^1) \\
&\qquad + ijN^2\psi^2 \cos(iNx^1)\cos(jNx^1) S(\nabla x^1,\nabla x^1)\Bigr\}
\end{split}
\end{equation}
Since $|\sin| \leq 1$ and $|\cos| \leq 1$,
\begin{equation}\label{pt2}
\begin{split}
&\text{(the left hand side of (\ref{pt}))}\\
\geq& -\sum_{i,j=K+1}^{K+k} \left(|a^i||a^j|C_3(1+iN+jN)\right)\\
&\qquad +\left(\sum_{i=K+1}^{K+k} iNa^i\psi \cos(iNx^1)\right)^2 S(\nabla x^1,\nabla x^1)\\
\geq& -3k(K+k)N|a|^2 C_3\\ 
&\qquad+ \Bigl(\sum_{i=K+1}^{K+k} iNa^i\psi \cos(iNx^1)\Bigr)^2 S(\nabla x^1,\nabla x^1)\\
\geq& -3k(K+k)N|a|^2 C_3 + \Bigl(\sum_{i=K+1}^{K+k} iNa^i\psi \cos(iNx^1)\Bigr)^2 \delta.
\end{split}
\end{equation}
By integrate both sides of (\ref{pt2}), we have
\begin{equation}\label{int}
\begin{split}
&\int_M S\left(\nabla\left(\sum_{i=K+1}^{K+k} \psi a^i u_i\right),\nabla\left(\sum_{j=K+1}^{K+k}\psi a^j u_j\right)\right)\,d\mu_g \\
\geq& \int_{[-\frac{2\pi}{N},\frac{2\pi}{N}]^n} \Bigl(-3k(K+k)N|a|^2 C_3 
+ \Bigl(\sum_{i=K+1}^{K+k} iNa^i\psi \cos(iNx^1)\Bigr)^2 \delta \Bigr)\,d\mu_g\\
\geq& -\int_{[-\frac{2\pi}{N},\frac{2\pi}{N}]^n} 3k(K+k)N|a|^2 C_2 C_3 \,d\mu_{\mathbb{R}}\\
&\quad+\int_{[-\frac{\pi}{N},\frac{\pi}{N}]^n} \Bigl(\sum_{i=K+1}^{K+k} iNa^i \cos(iNx^1)\Bigr)^2 \delta C_1 \,d\mu_{\mathbb{R}}\\
=& -3k(K+k)N|a|^2 C_2 C_3 \Bigl(\frac{4\pi}{N}\Bigr)^n\\ 
&\quad+ \sum_{i,j=K+1}^{K+k} ijN^2 a^i a^j \delta C_1 \Bigl(\frac{2\pi}{N}\Bigr)^{n-1} \int_{-\frac{\pi}{N}}^{\frac{\pi}{N}} \cos(iNx^1)\cos(jNx^1)\,dx^1\\
=& -3k(K+k)N|a|^2 C_2 C_3 \Bigl(\frac{4\pi}{N}\Bigr)^n 
+ \sum_{i,j=K+1}^{K+k} ijN^2 a^i a^j \delta C_1 \Bigl(\frac{2\pi}{N}\Bigr)^{n-1} \frac{\pi}{N}\delta_{ij}\\
\geq&|a|^2\Bigl( -3k(K+k)N C_2 C_3 \Bigl(\frac{4\pi}{N}\Bigr)^n 
+K^2 N^2 \delta C_1 \Bigl(\frac{2\pi}{N}\Bigr)^{n-1} \frac{\pi}{N}\Bigr).
\end{split}
\end{equation}

For any $k\in \mathbb{N}$ there exists a positive integer $K\in{N}$ such that 
\begin{equation*}
-3k(K+k)N C_2 C_3 \Bigl(\frac{4\pi}{N}\Bigr)^n 
+K^2 N^2 \delta C_1 \Bigl(\frac{2\pi}{N}\Bigr)^{n-1} \frac{\pi}{N}>0,
\end{equation*}
and so we have $\int_M S(\nabla v,\nabla v)\,d\mu_g>0$ for any $v \in V_{k,K} \backslash \{0\}$ by (\ref{int}).

We define $V_k$ as the image of the linear map $V_{k,K}\to C^\infty(M) \cap H, v\mapsto v-\frac{1}{\Vol(M)}\int_M v \,d\mu_g$.
Since the Kernel is $\{0\}$, $V_k$ is $k$-dimensional.
Moreover, for any $v \in V_k \backslash \{0\}$, we have $\int_M S(\nabla v,\nabla v)\,d\mu_g>0$; therefore $\Lambda_S(v)>0$.
\end{proof}

\bibliographystyle{amsbook}

\begin{thebibliography}{99}
\bibitem{bo}
R. L. Bishop, B. O'Neill, 
\textit{Manifolds of Negative curvature,}
Trans. Amer. Math. Soc. 145 (1969), 1--49.
\bibitem{gw}
C. S. Gordon, E. N. Wilson,
\textit{The spectrum of the Laplacian on Riemannian Heisenberg manifolds,}
Michigan Math. J. 33 (1986), 253--271.
\bibitem{gr}
J. F. Grosjean, 
\textit{A new Lichnerowicz-Obata estimate in the presence of a parallel p-form,}
Manuscripta Math. 107 (2002), no. 4, 503--520.
\bibitem{Kh} W. K\"{u}hnel,
\textit{Conformal transformations between Einstein spaces,}
Conformal geometry (Bonn, 1985/1986), Aspects Math., E12, vieweg, Braunschweig (1988), 105--146.
\bibitem{Li} A. Lichnerowicz,
\textit{G\'{e}om\'{e}trie des groupes de transformations,}
Travaux et Recherches Math\'{e}matiques, III. Dunod, Paris (1958).
\bibitem{Ob} M. Obata, 
\textit{Certain conditions for a Riemannian manifold to be isometric with a sphere,}
J. Math. Soc. Japan 14  (1962), 333--340.
\bibitem{S} M. Sugiura,
\textit{Fourier series of smooth functions on compact Lie groups,}
Osaka J. Math. 8 (1971), 33--47.
\bibitem{T1} Y. Tashiro, 
\textit{Complete Riemannian manifolds and some vector fields,}
Trans. Amer. Math. Soc. 117, (1965) 251--275.
\bibitem{T2} Y. Tashiro, 
\textit{Conformal transformations in complete Riemannian manifolds,}
Publ. the Study Group of Geometry, Vol. 3 (1967).
\bibitem{Tan}S. Tanno, 
\textit{Eigenvalues of the Laplacian of Riemannian manifolds,}
T\^{o}hoku Math. J. (2) 25, (1973), 391--403.
\bibitem{Ur}
H. Urakawa,
\textit{On the least positive eigenvalue of the Laplacian for compact group manifolds,}
J. math Soc .Japan 31, No. 1 (1979), 209--226.
\end{thebibliography}

\end{document}